\documentclass[11pt]{article}
\usepackage[T1]{fontenc}
\usepackage{lmodern}
\usepackage[margin=1in]{geometry}
\usepackage{microtype}
\usepackage{amsmath,amssymb,amsthm,mathtools,mathrsfs}
\usepackage{enumitem}
\usepackage{xcolor}

\usepackage{graphicx}
\usepackage[hidelinks]{hyperref}
\numberwithin{equation}{section}
\newtheorem{theorem}{Theorem}[section]
\newtheorem{proposition}[theorem]{Proposition}
\newtheorem{lemma}[theorem]{Lemma}
\newtheorem{corollary}[theorem]{Corollary}
\theoremstyle{definition}
\newtheorem{definition}[theorem]{Definition}

\theoremstyle{remark}
\newtheorem{remark}[theorem]{Remark}
\newcommand{\cA}{\mathcal{A}}
\newcommand{\cC}{\mathcal{C}}
\newcommand{\cD}{\mathcal{D}}
\newcommand{\cE}{\mathcal{E}}
\newcommand{\cF}{\mathcal{F}}
\newcommand{\cL}{\mathcal{L}}
\newcommand{\cM}{\mathcal{M}}
\newcommand{\cN}{\mathcal{N}}
\newcommand{\cQ}{\mathcal{Q}}
\newcommand{\cR}{\mathcal{R}}
\newcommand{\cS}{\mathcal{S}}
\newcommand{\cT}{\mathcal{T}}
\newcommand{\cV}{\mathcal{V}}
\newcommand{\wh}{\widehat}
\newcommand{\wt}{\widetilde}
\newcommand{\ol}{\overline}
\newcommand{\D}{\mathbb D}
\newcommand{\R}{\mathbb R}
\newcommand{\C}{\mathbb C}
\renewcommand{\P}{\mathbb P}
\newcommand{\E}{\mathbb E}

\newcommand{\dd}{\,d}
\newcommand{\Smeas}{\mathsf S}
\newcommand{\Ameas}{\mathsf A}

\newcommand{\diskone}{{\mathsf D_{\scriptscriptstyle\bullet}}}

\DeclareMathOperator{\Var}{Var}

\DeclareMathOperator{\sgn}{sgn}
\DeclareMathOperator{\PV}{PV}

\newcommand{\Hol}{\operatorname{Hol}}
\newcommand{\q}{q_*}
\renewcommand{\Re}{\operatorname{Re}}
\renewcommand{\Im}{\operatorname{Im}}

\title{JT gravity as a conformal scaling limit \\ of random flat combinatorial disks}
\author{Morris Ang\thanks{University of California San Diego}}
\date{}
\begin{document}
\maketitle

\begin{abstract}
	We prove that flat Jackiw--Teitelboim (JT) gravity arises as a conformal scaling limit of random flat combinatorial disks. We give a clean criterion, formulated in terms of branching explorations, for a discrete model to converge to this conjecturally universal limit. For simplicity, we apply this criterion to a Boltzmann ensemble of disks with rectilinear boundary, defined by summing over combinatorial types and integrating over lengths. As the length of a distinguished side tends to infinity, the rescaled surface converges in law to a continuum random surface prescribed by Ferrari's path-integral formulation of JT gravity.
	Convergence is in the topology of smooth convergence of the conformal factors on compact subsets of the interior when uniformizing in a fixed domain; the limiting conformal factor is described by a harmonic Gaussian field closely related to holomorphic multiplicative chaos.

	The crux is a branching exploration of flat JT gravity whose lengths evolve as conditioned Cauchy processes and whose unexplored components are conditionally independent. We construct this exploration using a conformal welding theorem proved in a companion paper, and show that it is the scaling limit of Budd’s row-by-row exploration of the discrete disks, analogous to the peeling of random planar maps. Under a tightness condition, our criterion upgrades this convergence to a conformal scaling limit for the discrete disks. Other natural combinatorial disk models admit analogous first-row recursions, providing a route to their scaling limits.
\end{abstract}

\tableofcontents

\section{Introduction}
\label{sec-intro}

A \emph{rectilinear disk} is an oriented closed topological disk equipped with a Riemannian metric which is flat on the interior, whose boundary comprises finitely many geodesic segments meeting at interior angles $\pi/2$ or $3\pi/2$. We call these
\emph{convex} and \emph{concave corners} respectively, and distinguish a side with convex endpoints as the \emph{base}. These disks admit isometric immersions into the plane, which
may overlap themselves as in Figure~\ref{fig-boltzmann-disks}.

We study a model of random rectilinear disks based on Budd's work \cite{budd-flat-disks}, where each rectilinear disk $\cS$ is assigned the weight $w(\cS) = (4\pi)^{-c(\cS)/2} e^{-r(\cS)/2}$. Here, $c(\cS)$ is the number of corners of $\cS$, and $r(\cS)$ is the perimeter of $\cS$ excluding the base. These are combinatorial disks with continuous lengths:  each disk can be subdivided into rectangles by drawing horizontal and vertical rays inward from each concave corner, and is determined by the combinatorial type of this subdivision together with the side lengths. The law of the \emph{Boltzmann rectilinear disk} is obtained by summing over combinatorial types and integrating over these lengths with weight $w$. It is critical in that its partition function is finite but would diverge if the corner weight $(4\pi)^{-1/2}$ were increased.   

\begin{figure}[ht!]
	\begin{center}
		\includegraphics[scale=0.235]{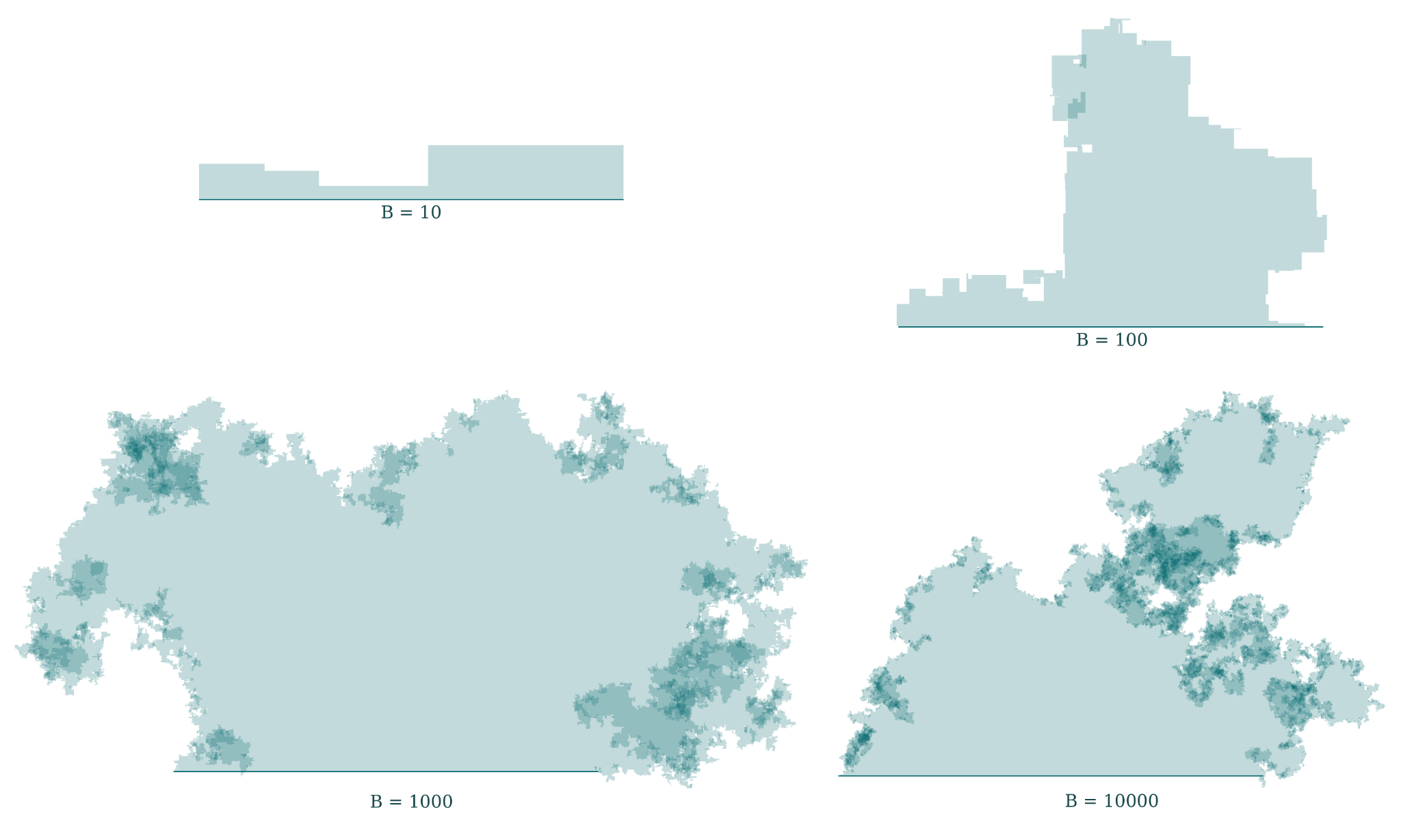}
	\end{center}
	\caption{\label{fig-boltzmann-disks} Simulations of Boltzmann rectilinear disks with base length $B$. These surfaces are isometrically immersed in the plane, and areas with darker shades depict overlaps.} 
\end{figure}

Ferrari \cite{ferrari-jt,ferrari-lattice} proposed that suitable models
of random discrete flat disks should converge to a continuum random surface
described by \emph{flat Jackiw--Teitelboim (JT) gravity}, whose random flat metric is described by the exponential of a harmonic Gaussian field (flatness is equivalent to harmonicity of the conformal factor). We consider a  variant of this surface which has a straight boundary segment: equip $\D_+ = \{ z = re^{i \theta} \: : \: 0 < r < 1  \text{ and } 0 < \theta < \pi\}$ with the Riemannian metric  \begin{equation}\label{eq-jt-metric}
	\frac{\exp\!\left(2\sum_{n\ge1}\frac1{\sqrt n} G_n r^n\cos(n\theta)\right)}
	{|1-z^2|}\,|dz|^2,
	\qquad G_n\overset{\mathrm{i.i.d.}}{\sim}N(0,1),
\end{equation}
distinguish $[-1,1]$ as its base, and rescale all distances so the base has length one. Denote the resulting surface by $\cS$. See Figure~\ref{fig-jt-sectors}.

\begin{figure}[ht!]
	\begin{center}
		\includegraphics[scale=0.23]{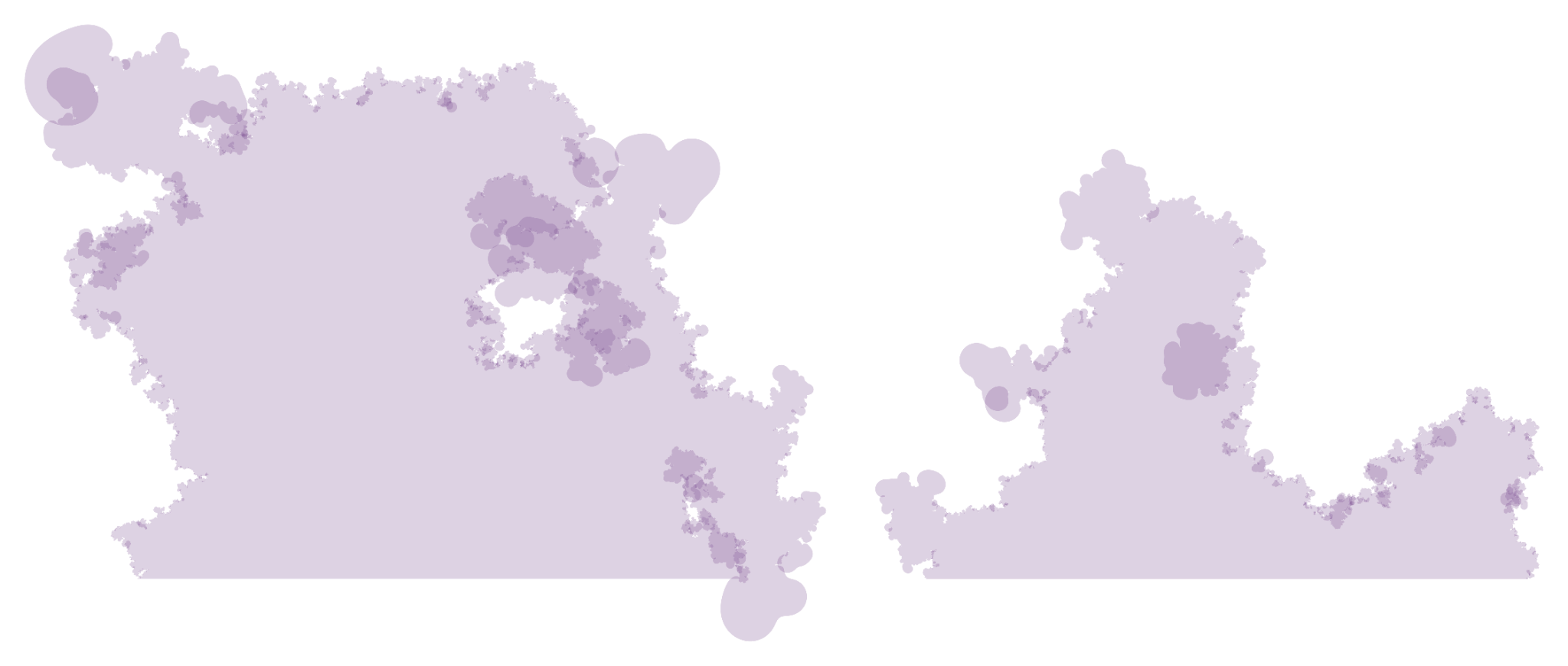}
	\end{center}
	\caption{\label{fig-jt-sectors} Two simulations of angle-$\pi$ JT sectors (Definition~\ref{def-jt-sector}) isometrically immersed in the plane. Darker shades depict overlaps. The Gaussian series is truncated at $10^7$ terms.} 
\end{figure}

 Theorem~\ref{thm-main} shows that as $B \to \infty$, the critical Boltzmann rectilinear disk $\cS_B$ with base length $B$ converges in distribution after rescaling to the JT gravity surface  $\cS$.
More precisely, uniformize $\cS_B$ and $\cS$ in
$\D_+$, sending each base to $[-1,1]$ and its midpoint to $0$.
Writing their metrics as $e^{2\sigma_B(z)}|\dd z|^2$ and
$e^{2\sigma(z)}|\dd z|^2$, respectively, we have 
\[
\sigma_B-\log B \Longrightarrow \sigma
\qquad\text{as }B\to\infty,
\]
in the topology of uniform convergence on compact subsets of $\D_+\cup(-1,1)$.

More generally, Theorem~\ref{thm-exploration-criterion} provides a clean criterion where one must check three properties of an exploration of a discrete model: (a) a Markov property, (b) tightness, (c) joint convergence of a length process and associated data, with the limiting length process given by a Doob $h$-transform of a pair of independent Cauchy processes. Given these, the scaling limit of the discrete model is flat JT gravity. We expect that these properties hold for a large family of discrete flat disks (e.g., flat triangulations or quadrangulations), so Theorem~\ref{thm-exploration-criterion} should provide a road to establishing  flat JT gravity as the universal scaling limit of discrete flat disks. 

There is a parallel story in the setting of \emph{random planar maps}, which are discrete random surfaces without a flatness condition imposed. Le Gall \cite{le-gall-brownian-map} and Miermont \cite{miermont-brownian-map} independently proved that certain classes of uniform random planar maps, viewed as metric measure spaces, converge in distribution after rescaling to the \emph{Brownian map} \cite{mm-brownian-map}. Subsequent work of many authors established convergence to the same limit for many other classes of planar maps, demonstrating universality of the Brownian map; see e.g.\ \cite{le-gall-brownian-geometry} for a survey. Miller and Sheffield later identified the Brownian map with the metric measure space associated to $\sqrt{8/3}$-Liouville quantum gravity (LQG), i.e., the random geometry associated to the formal metric ``$e^{\sqrt{8/3} h(z)}|\dd z|^2 $'' where $h$ is a Gaussian free field \cite{lqg-tbm1,lqg-tbm2}. Uniform triangulations \cite{hs-cardy-embedding} and spanning-tree-weighted
planar maps \cite{holden-tutte-embedding,hy-conformal-embeddings} were later shown to converge to LQG under conformal embeddings.
These results connect natural discrete models with a continuum limit predicted by the two-dimensional quantum gravity literature. Our work establishes an analogous correspondence for Boltzmann rectilinear disks, whose continuum limit is defined using the \emph{harmonic part} of a Gaussian free field.

Our proof depends crucially on an input from our companion paper \cite{ang-jt-zipper}: cutting a JT gravity surface along a flow line yields surfaces which are  conditionally independent given the length of the cut. 
This is inspired by Sheffield's quantum zipper
\cite{shef-zipper} and its extensions \cite{wedges, ahs-disk-welding}, which give analogous decomposition results
for Liouville quantum gravity surfaces cut by independent
Schramm--Loewner evolution curves.
In the JT gravity setting, however, the cut is smooth and measurable with respect to the surface.
Using this cutting identity, we construct a branching exploration
of $\cS$ whose lengths evolve as conditioned Cauchy
processes and whose unexplored regions are conditionally
independent given their base lengths. This continuum exploration is the
counterpart of the row-by-row exploration of discrete surfaces introduced by Budd \cite{budd-flat-disks}.
 Comparing discrete and continuum  branching explorations yields Theorem~\ref{thm-exploration-criterion}.

\subsection{Boltzmann rectilinear disks and their scaling limit}\label{subsec-boltzmann-disks}

Any rectilinear disk can be decomposed by drawing from each concave corner a horizontal and a vertical ray into the surface, stopped when they hit the boundary; cutting along all these rays gives a collection of rectangles. These rectangles can be viewed as the faces of a planar map. This is a quadrangulation with the disk topology and with simple boundary, such that interior vertices have degree exactly four (hence are \emph{flat}), and boundary vertices have degrees in $\{ 2, 3, 4\}$ (corresponding respectively to convex corners, straight boundary vertices, concave corners). Given this planar map and the positive real lengths associated to the sides, one can recover the rectilinear disk.

Following Budd's description of rectilinear disks and their moduli space \cite[Section~1.3]{budd-flat-disks}, we define a Lebesgue measure on the space of rectilinear disks. Generically, each ray emanating from a concave corner terminates at a straight boundary vertex. This is precisely the condition characterizing the \emph{rigid quadrangulations} introduced in \cite{budd-flat-disks}. For a rigid quadrangulation $\cR$ with $c$ corners and a distinguished side with convex endpoints called its \emph{base}, the rectilinear disks having combinatorial type $\cR$ are parametrized by the $c-2$ widths of its rows and columns. The pushforward of Lebesgue measure on $(0,\infty)^{c-2}$ to the space of rectilinear disks under this map yields a measure $\lambda_{\cR}$. The Lebesgue measure on the space of rectilinear disks is then $\lambda = \sum_{\cR} \lambda_{\cR}$ where the sum is taken over the set of all rigid quadrangulations. Disintegrate with respect to base length, so  $\lambda = \int_0^\infty \lambda^b \dd b$ where each measure $\lambda^b$ is supported on the set of rectilinear disks with base length $b$.

Recall that $c(\cS)$ is the number of corners of
$\cS$ and $r(\cS)$ is its perimeter excluding
the base. For $b>0$, define 
\[
\q=(4\pi)^{-1/2},\qquad 
Z(b)=\int \q^{c(\cS)}e^{-r(\cS)/2}
\,\lambda^b(\dd\cS).
\]
The partition function $Z(b)$ is finite by Lemma~\ref{lem-partition}.
We define a
\emph{Boltzmann rectilinear disk} with base length $b$
to be a random rectilinear disk with law
\begin{equation}\label{eq-boltzmann-disk-law}
	Z(b)^{-1}\q^{c(\cS)}e^{-r(\cS)/2}
	\,\lambda^b(\dd\cS).
\end{equation}


To construct the limiting surface, 
equip $\D_+$ with the Riemannian metric described in \eqref{eq-jt-metric}. Let $[-1,1]$ be its base, and rescale the surface so it has base length 1. 
Denote the resulting surface by $\cS$.

\begin{theorem}\label{thm-main}
	Let $\cS_B$ be a Boltzmann rectilinear disk with base length $B$, and let $\cS$ be sampled as above. Conformally parametrize $\cS_B$ and $\cS$ in $\D_+$ so their bases are sent to $[-1,1]$ with midpoints sent to $0$. In these coordinates, write their metrics as
	$e^{2\sigma_{\cS_B}(z)}|\dd z|^2$ and
	$e^{2\sigma_{\cS}(z)}|\dd z|^2$.
Then as $B\to\infty$, 
	\begin{equation}\label{eq-main}
		\sigma_{\cS_B} - \log B \ \Longrightarrow\ \sigma_{\cS}
	\end{equation}
	with respect to uniform convergence on compact subsets of $\D_+ \cup (-1, 1)$. 
\end{theorem}
We prove Theorem~\ref{thm-main} by verifying the hypotheses of
the convergence criterion Theorem~\ref{thm-exploration-criterion} 
for Boltzmann rectilinear disks, using arguments strongly inspired by
\cite{bck-planar-maps, bbck-martingales, ads-growth-fragmentation}.

\subsection{Branching explorations and a general convergence criterion}\label{subsec-exploration-criterion}

Flat JT gravity is expected to be the universal scaling limit for discrete flat disks. For example, Ferrari predicts that models of flat disks tiled by squares, triangles, or hexagons yield the same continuum theory \cite[Sections~I,D and~II,B]{ferrari-lattice}. Theorem~\ref{thm-exploration-criterion} provides a route to establishing such universality: for models with a branching Markov property, convergence of the data of a single vertical exploration and tightness of maximal horizontal widths imply conformal convergence to JT gravity.

For Boltzmann rectilinear disks, the exploration comes from Budd's
row-by-row decomposition, analogous to peeling random planar maps \cite{budd-flat-disks}. 
The central new ingredient is a branching exploration of JT gravity, which we construct and identify as the scaling limit of the discrete exploration.

We now formulate the criterion precisely.
A \emph{based flat polygonal disk} is an oriented closed
topological disk equipped with a Riemannian metric that is flat
in the interior, with piecewise geodesic boundary having finitely
many vertices, all of which have angle less than $2\pi$, and a
distinguished side called its \emph{base}. We identify disks under
isometries preserving the orientation and the base.

Let $\cS$ be a based flat polygonal disk with base length $b$, and fix
$0<\delta<b$. Choose a uniform point on its base and follow the
inward vertical ray $(p_t)_{0\le t<T}$, parametrized by length,
where $T$ is its first boundary hitting time.

Through each $p_t$, extend a horizontal segment in both directions
until further straight extension would leave $\cS$.
The segment may pass through boundary vertices and run
along horizontal boundary sides, provided it continues straight
and remains in $\cS$. For each $t\in(0,T)$ let $(L_t,R_t)$
be the lengths of the horizontal segments to the left
and right of $p_t$, replace $(L_t,R_t)_{[0,T)}$ with its
right-continuous version, and set $L_T=R_T=0$.
Stop the exploration at the first time $\tau_\delta\leq T$ that
$L_{\tau_\delta}+R_{\tau_\delta}\leq\delta$.

For each $t\in[0,\tau_\delta]$, cut $\cS$ along the horizontal
segment through $p_t$. The resulting pieces whose interiors
do not meet the stopped ray $\{p_s:0<s<\tau_\delta\}$ are called the
\emph{children} born at time $t$. Each child is regarded as a
based flat polygonal disk with its base along the horizontal
cut, and the \emph{explored region} is the closure in $\cS$
of the complement of the union of all children.
The adjacency graph of the children and the explored region,
with an edge for each attachment interval, is a tree, so each
child is connected to the explored region by a unique chain.

We represent each child by a triple $(I,J,\varepsilon)$:
view the starting point of the exploration as the origin of
the complex plane and the vertical ray as a subset of the
positive imaginary axis. Let $J\subset\C$ be the full
base of the child, and let $\varepsilon$ be $+1$ if the child
locally lies above its base and $-1$ otherwise.
Let $I\subseteq J$ be its attachment interval to the neighboring
piece on the unique chain of attachments leading to the explored
region. This neighboring piece may be the explored region itself
or another child.
Let $\cC_\delta$ be the unordered collection of triples
	$(I,J,\varepsilon)$ for children whose chain to the explored
	region consists entirely of attachment intervals of length
	greater than $\delta$. See Figure~\ref{fig-exploration}.

\begin{figure}[ht!]
	\begin{center}
		\includegraphics[scale=0.75]{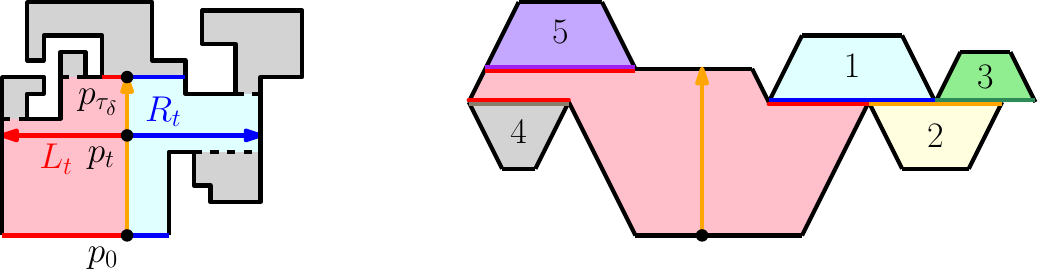}
	\end{center}
	\caption{\label{fig-exploration} 
		\textbf{Left:} The vertical ray $(p_t)_{[0,\tau_\delta]}$ (orange) is run until time $\tau_\delta$. From each point $p_t$, we draw out a pair of maximal horizontal segments (red, blue) and let their lengths be $L_t, R_t$. These horizontal segments sweep out the explored region (pink, light blue), and each connected component of the unexplored region is a child (grey).  \textbf{Right:} The explored region is in pink. For each child, we record its base $J$ (same color as the child) and the orientation $\varepsilon$ (here $\varepsilon_1 = \varepsilon_3 = \varepsilon_5 = 1$ and $\varepsilon_2 = \varepsilon_4 = -1$). The intervals $I_1, I_4, I_5$ are the respective red segments, whereas $I_2 = J_2 \cap J_1$ and $I_3 = J_3 \cap J_2$. If we had $|I_2| \leq \delta$ and $|I_k| > \delta$ for $k \neq 2$, then $\cC_\delta=\{ (I_k, J_k, \varepsilon_k) \::\: k = 1,4,5\}$. The data of Definition~\ref{def-vertical-exploration-data} describes the shape of the explored region via $(L_t, R_t)_{[0,\tau_\delta]}$, and records  information $\cC_\delta$ on all children that remain after discarding children behind $\delta$-short attachment intervals:
		 children 2 and 3 are both excluded because their attachment
		chains pass through $I_2$.} 
\end{figure}

\begin{definition}
	\label{def-vertical-exploration-data}
	The data of a vertical exploration stopped at cutoff $\delta$ is
	$\bigl((L,R)_{[0,\tau_\delta]},\cC_\delta\bigr)$.
	Convergence of data means $\tau_\delta^{(n)}\to\tau_\delta$,
	\[
	\bigl(
	L^{(n)}(\tau_\delta^{(n)}t),
	R^{(n)}(\tau_\delta^{(n)}t)
	\bigr)_{t\in[0,1]}
	\longrightarrow
	\bigl(L(\tau_\delta t),R(\tau_\delta t)\bigr)_{t\in[0,1]}
	\]
	in the joint Skorokhod $J_1$ topology,
	and an eventual bijection between $\cC_\delta^{(n)}$ and
	$\cC_\delta$ under which the triples converge.
\end{definition}

For $b>0$, let $\P_b$ be the law of the limiting
	surface $\cS$ from Theorem~\ref{thm-main} with all distances
	multiplied by $b$, so its base has length $b$.
	The analogous vertical exploration data for a sample from
	$\P_b$, stopped at cutoff $0<\delta<b$, are defined in
	Definition~\ref{def-jt-exploration-data}.
	Its length process is a Doob $h$-transform of a pair of independent
	standard symmetric Cauchy processes killed on leaving
	$(0,\infty)^2$, with $h(\ell,r)=(\ell+r)^{-2}$, and
	$\cC_\delta$ is a measurable function of the stopped
	length path.

A family $(\mathsf Q_b)_{b>0}$ of laws on based flat polygonal disks
of base length $b$ has the \emph{branching Markov property}
if for every $b>0$ and $0<\delta<b$, a sample from
	$\mathsf Q_b$ has the following property:
conditionally on the vertical exploration data
(Definition~\ref{def-vertical-exploration-data}), the children indexed by
$\cC_\delta$ are independent, each with law
$\mathsf Q_c$ where $c$ is the length of its base.

\begin{theorem}\label{thm-exploration-criterion}
	For each $n$, let $(\mathsf Q_b^{(n)})_{b>0}$ be a measurable
	family of laws on based flat polygonal disks of base length $b$,
	which satisfies the branching Markov property.
	Let $\cS^{(n)}\sim\mathsf Q_1^{(n)}$ and
		$\cS\sim\P_1$.
	Suppose that:
	\begin{enumerate}[label=\textup{(\roman*)}]
		\item Whenever $b_n\to b>0$ and $0<\delta<b$, we have
		$\bigl(
		(L^{(n)},R^{(n)})_{[0,\tau_\delta^{(n)}]},
		\cC_\delta^{(n)}
		\bigr)
		\Longrightarrow
		\bigl((L,R)_{[0,\tau_\delta]},\cC_\delta\bigr)$,
		where the tuples are the vertical exploration data of
		$\mathsf Q_{b_n}^{(n)}$
		(Definition~\ref{def-vertical-exploration-data}) and $\P_b$
		(Definition~\ref{def-jt-exploration-data}) stopped at cutoff $\delta$.
		\item The sequence $(W_{\max}^{(n)})$ is tight, where
		$W_{\max}^{(n)}$ is the supremum of the lengths of horizontal segments in the interior of 
		$\cS^{(n)}$ (with horizontal meaning parallel
		to the base).
	\end{enumerate}
	Conformally parametrize these surfaces in
	$\D_+$ so their bases are sent to $[-1,1]$ with midpoints
	sent to $0$.
	In these coordinates, write their metrics as
	$e^{2\sigma_{\cS^{(n)}}(z)}|\dd z|^2$ and
	$e^{2\sigma_{\cS}(z)}|\dd z|^2$, respectively.
	Then their conformal factors satisfy
	\[
	\sigma_{\cS^{(n)}}\Longrightarrow\sigma_{\cS}
	\]
	with respect to uniform convergence on compact subsets of
	$\D_+ \cup (-1,1)$.
\end{theorem}

\medskip
\noindent \textbf{Proof strategy.}
The crux is to construct and study a branching exploration
of JT gravity that is the continuum counterpart of Budd's
row-by-row exploration of discrete surfaces \cite{budd-flat-disks}.

\smallskip
\noindent\emph{Identifying the length process.}
The key input is the conformal welding identity Theorem~\ref{thm-sector-cutting} from \cite{ang-jt-zipper} for random flat surfaces
called \emph{JT sectors}: cutting an angle-$\pi$ JT sector along
its vertical flow line produces two angle-$\pi/2$ JT sectors
that are conditionally independent given the length of the cut. 
Combining this with a
rerooting symmetry of the angle-$\pi$ sector (Proposition~\ref{prop-reroot}), we show that after also cutting along the horizontal chord through $p_t$, the two pieces above the chord are angle-$\pi/2$ JT sectors. Since we can identify the angle-$\pi/2$ partition function with a Cauchy
excursion entrance density
(Proposition~\ref{prop-quarter-identification}), this yields
the length process of the vertical exploration
(Proposition~\ref{prop-retained}).
We also establish a strong Markov property for the
to-be-explored region (Proposition~\ref{prop-stopped}).

\smallskip
\noindent\emph{Constructing the branching exploration.}
To obtain a branching exploration, we must also identify
the conditional laws of the children created at jumps of the length process.
Rotation and rerooting symmetries (Proposition~\ref{prop-glued-symmetry}) allow us to view each child
as the to-be-explored region of another vertical exploration
started outside the child, 
and the strong Markov property then identifies its conditional
law given the rest of the surface (Lemma~\ref{lem-daughter-law}).
This yields conditional independence of the children
(Corollary~\ref{cor-daughter-history}), allowing the vertical exploration
to be iterated in each child. This gives the branching exploration. 
\smallskip 

Finally,  in Section~\ref{sec-developing-convergence},  we prove Theorem~\ref{thm-exploration-criterion} by reconstructing explored
regions from (i) and using tightness from (ii) to control their immersions in the plane.

\subsection{Related work}\label{subsec-related-work}
\noindent \textbf{Jackiw--Teitelboim gravity.}
JT gravity is a two-dimensional dilaton
gravity theory in which the dilaton imposes constant bulk curvature.
It has attracted considerable attention as a tractable setting for studying black holes and holography; see \cite{mt-jt-review} for a review.
Much of the literature concerns the negative-curvature theory
with asymptotically $\mathrm{AdS}_2$ boundary conditions (meaning the boundary is at infinite distance).
In this setting, 
the genus expansion of
the partition functions of pure JT gravity agrees with that
of a double-scaled Hermitian matrix integral \cite{sss-matrix-integral}.
The present work instead concerns the zero-curvature theory at finite
cutoff (meaning the boundary is at finite distance), describing random flat disks with fluctuating boundaries. Stanford and Yang \cite{sy-self-avoiding-loops} study flat JT gravity via
self-avoiding loops, whereas our discrete models and their scaling
limit allow self-intersections of the developed boundary.

\medskip
\noindent\textbf{Ferrari's formulation of flat JT gravity.}
Ferrari \cite{ferrari-jt,ferrari-lattice} proposed discrete and
continuum formulations of JT gravity.
In the flat case, he primarily studies disks tiled by squares,
with four squares incident to each interior vertex.
He proposes similar constructions using regular triangles or hexagons, and conjectures universality: these
discretizations should yield the same continuum theory
\cite[Sections~I\,D and~II\,B]{ferrari-lattice}.
He also constructs a matrix model encoding the enumeration of square-tiled disks \cite[Section~V]{ferrari-lattice}.

Ferrari's continuum path integral depends on the matter central
charge $c$.
The pure-gravity case $c=0$ is the conjectured scaling limit of
the discrete models above; the metric in \eqref{eq-jt-metric}
describes a version with a distinguished straight boundary segment.
Using this path integral, Ferrari \cite{ferrari-jt} predicts
critical exponents as functions of $c$; see also \cite{cf-one-loop} for an analysis of the
semiclassical limit $c\to-\infty$.
At $c=0$, these predictions motivate conjectures for the discrete
models: boundary dimension two and a power-law prefactor with  exponent $-3$
for the number of disks with a given perimeter
\cite[Conjectures~4.1--4.2]{ferrari-lattice}.

\medskip
\noindent \textbf{Budd's rigid quadrangulations.}
Motivated by Ferrari's proposal, Budd introduced rigid quadrangulations and established a bijection with colorful quadrangulations \cite{budd-flat-disks}.
The latter have no flatness constraint and have integer vertex labels that differ by one along each edge and take three distinct values around each face.
This bijection transfers enumeration results on colorful quadrangulations
\cite{bmep-eulerian-orientations} to rigid quadrangulations, in agreement with Ferrari's predicted  $-3$ exponent.
The bijection also identifies the row-by-row exploration of rigid
quadrangulations with a peeling exploration of colorful
quadrangulations. \cite{zonneveld-tree-bijection} gives a direct bijective proof
of the generating functions of rigid quadrangulations
by encoding a variant of the row-by-row decomposition
in decorated binary trees. \cite{budd-flat-disks} further suggests that these bijections may have a continuum
counterpart relating flat JT gravity to critical LQG with an independent
$\mathrm{CLE}_4$.

\medskip
\noindent \textbf{Branching explorations.}
Branching explorations of random planar maps via peeling and their
growth-fragmentation scaling limits were developed in
\cite{bck-planar-maps, bbck-martingales}.
A signed variant of one such growth-fragmentation process was realized
in \cite{ads-growth-fragmentation} by cutting a planar Brownian excursion at
horizontal levels and recording the signed horizontal displacements
of the resulting subexcursions.
The signed base lengths in our branching exploration of JT gravity, with sign corresponding to the orientation of the child, 
evolve according to  this growth-fragmentation.
See also \cite{da-silva-signed-growth-fragmentations} for the general theory of signed
growth-fragmentations, or \cite{bcr-markov-trees, chq-critical-markov-trees}
for the broader theory of self-similar Markov trees.
The same signed growth-fragmentation appears in the critical
mating-of-trees theorem \cite[Section~5.2.2]{ahps-critical-lqg} describing critical
LQG coupled with an independent conformal loop ensemble, consistent
with the connection suggested by \cite{budd-flat-disks}.

\medskip
\noindent \textbf{Holomorphic multiplicative chaos.}
As explained in \eqref{eq-developing-map}, a flat surface
$(\D,e^{2\sigma(z)}|\dd z|^2)$ admits a holomorphic
\emph{developing map} $F:\D\to\C$ satisfying
$|F'|^2=e^{2\sigma}$.
In our JT gravity setting, these maps are obtained by integrating
a half-disk variant of critical \emph{holomorphic multiplicative
	chaos} (HMC).

HMC is obtained by exponentiating a Gaussian holomorphic field
on $\D$ with logarithmic boundary correlations.
In our variant, this field has real Gaussian coefficients (hence has reflection symmetry across the real axis).
For the standard model with complex Gaussian coefficients,
Duplantier--Gayrard--Saksman
\cite[Theorem~1.2 and Proposition~1.4]{dgs-integral-means}
computed the almost sure integral means spectrum of the primitive
at criticality and proved its almost sure noninjectivity.

A systematic study of HMC as a random distribution on the circle,
including sharp results on its Sobolev regularity, was developed
in \cite{nps-hmc}.
HMC has close connections with the asymptotics of characteristic
polynomials of circular random matrix ensembles
\cite{hko-characteristic-polynomial,dg-magic-squares,sw-riemann-zeta,cn-gmc-beta}.
Further work on its coefficient moments and limiting distributions
includes \cite{sz-number-theory-chaos,npsv-fourier-coefficients,an-critical-hmc}.

\medskip
\noindent \textbf{Conformal welding of JT sectors.}
The companion paper \cite{ang-jt-zipper}\footnote{At the time of posting,
	the companion paper is being revised and is not yet publicly
	available. We expect to make it available shortly.}
introduces flat surfaces called \emph{JT sectors} (Definition~\ref{def-jt-sector}) and proves a decomposition
(Theorem~\ref{thm-sector-cutting}) inspired by Sheffield's quantum zipper
for Liouville quantum gravity \cite{shef-zipper, wedges, ahs-disk-welding}:
cutting an angle-$\pi$ JT sector along its vertical flow line
produces two angle-$\pi/2$ JT sectors that are conditionally
independent given the cut length.
This identity is central to the proofs of our theorems.

\subsection{Outline}\label{subsec-outline}

In Section~\ref{sec-discrete-exploration}, we compute the partition function of
Boltzmann rectilinear disks and describe their row-by-row and
vertical explorations.
Section~\ref{sec-jt-sectors} introduces JT sectors, states the conformal
welding input, and computes the partition functions for angles
$\pi$ and $\pi/2$.
In Section~\ref{sec-vertical-exploration}, we construct the vertical exploration of a
JT sector, identify its length process as a pair of conditioned Cauchy processes, and
prove its strong Markov property.
In Section~\ref{sec-branching}, we establish the branching Markov
property and construct the full branching exploration.
In Section~\ref{sec-discrete-convergence}, we verify the hypotheses of
Theorem~\ref{thm-exploration-criterion} for Boltzmann rectilinear disks,
thus proving Theorem~\ref{thm-main}.
We prove Theorem~\ref{thm-exploration-criterion} in
Section~\ref{sec-developing-convergence} by comparing discrete and continuum branching processes. 
Finally, Appendix~\ref{app-exploration-geometry} supplies geometric
results about horizontal chords and exploration jumps used
in the continuum construction.

\medskip
\noindent \textbf{AI Disclosure.}
The main mathematical ideas and overall proof strategies in
this paper and \cite{ang-jt-zipper} are due to the author. The only AI-generated argument that the author regards as a novel contribution is the proof of \cite[Proposition A.1]{ang-jt-zipper}, stated in this paper as Proposition~\ref{prop-boundary-continuity} (the author does not claim credit for discovering this proof). In this paper, GPT-6 Astra sped up the project through drafting, revision, and simulations, as well as literature searches for proof ingredients of Lemma~\ref{lem-partition}, Lemma~\ref{lem-cauchy-gram}, Proposition~\ref{prop-bloch-width}, and arguments for the proofs of  Lemma~\ref{lem-mellin},  Lemma~\ref{lem-tagged-paths}, Proposition~\ref{prop-all-time-bases}, Lemma~\ref{lem-actual-minimum},  Proposition~\ref{prop-binary-geometric}, and lemmas in the appendix;
from the author's perspective, these uses of AI for proof assistance were for routine arguments.  
The author has checked all AI-generated arguments and takes responsibility for their correctness.

\medskip
\noindent \textbf{Acknowledgements.}
The author thanks Timothy Budd, Nicolas Curien, Frank Ferrari,
Janne Junnila, and Eero Saksman for helpful conversations. In particular, Curien pointed out that the GPT-6-generated proof of  Proposition~\ref{prop-boundary-continuity} (to be included in \cite{ang-jt-zipper}) draws
heavily on the argument of~\cite{ahs-discounted-tree-sums}. While finishing this work, the author learned of concurrent independent work of Budd and Curien focusing on scaling limits from the metric space perspective. We are considering writing a joint paper revising and expanding the present work.

The author also thanks the organizers of the program
``Random Geometry in Math and Physics'' at the Simons Center
for Geometry and Physics and the New York satellite meeting
of ``Probabilistic Paths to QFT,'' both held in 2026,
where he learned about the problems studied in this paper.
He is grateful to the Erwin Schr\"odinger Institute for Mathematics and Physics for its hospitality during the
program ``Statistical Mechanics and Combinatorics of Discrete
Planar Structures,'' where much of this paper was written.

The author was partially supported by NSF grant DMS-2348201
and a start-up grant from the University of California San Diego.

\section{Discrete branching exploration}
\label{sec-discrete-exploration}

For $B,b>0$, let $\P_b^{(B)}$ be the law obtained by
sampling a Boltzmann rectilinear disk with base length $Bb$ and dividing
all its lengths by $B$. Thus a sample from $\P_b^{(B)}$ has base
length $b$, and $\P_b^{(1)}$ is the unscaled law.

In Section~\ref{subsec-discrete-partition-function} we compute the {partition} function, building on \cite{budd-flat-disks,bmep-eulerian-orientations}. 
In Section~\ref{subsec-row-by-row} we work with $B=1$, and discuss the row-by-row exploration; this is a branching exploration of the Boltzmann rectilinear disk, analogous to the peeling exploration of random planar maps. It was first introduced by Budd \cite{budd-flat-disks} to study the closely-related model of rigid quadrangulations. In Section~\ref{subsec-discrete-vertical} we examine a single branch by iterating the row-by-row exploration, and show that a conditioned variant of a pair of independent compount Poisson processes arises naturally. The crux of this paper is to show the limiting conditioned Cauchy processes arise in the continuum JT gravity surface (Section~\ref{sec-vertical-exploration}).

\subsection{The partition function} \label{subsec-discrete-partition-function}
We now compute the partition function of the Boltzmann rectilinear disk.
\begin{lemma}\label{lem-partition}
	The partition function in \eqref{eq-boltzmann-disk-law} satisfies
	\begin{equation}\label{eq-partition}
		Z(b)=\frac1{8\pi}\int_0^1
		s\sqrt{\frac{1-s}{1+s}}e^{-bs/2}\dd s,
		\qquad Z(b)\sim\frac1{2\pi b^2}
		\quad\text{as }b\to\infty.
	\end{equation}
	Moreover,
	$\int_0^\infty Z(b)\dd b=\frac18-{\q^2}$.
\end{lemma}

\begin{proof}
	We first compute the contribution to $Z(b)$ of a fixed rigid quadrangulation $\cR$ whose base
	has $p$ edges. Recall that a flat surface $\cS$ is obtained from $\cR$ by specifying the widths of the rows and columns of $\cR$. 
	Let $w_1,\dots,w_p$ be the widths of the $p$ columns meeting
	the base, and let $v_1, \dots, v_k$ be the
	remaining row and column widths. The base length is
	$b = w_1+\cdots+w_p$, and the perimeter excluding the base has length $b + 2 \sum v_j$. Thus, the Boltzmann weight of $\cS$ is
	\[q_*^{c(\cS)} e^{-r(\cS)/2} = q_*^{c(\cR)} e^{-b/2} e^{-\sum v_j}.\]
	Integrating this against Lebesgue measure on $(v_1, \dots, v_k) \in (0,\infty)^{k}$ gives $q_*^{c(\cR)} e^{-b/2}$, 
	and the pushforward of Lebesgue measure on $(0,\infty)^p$
	under $(w_1,\dots,w_p)\mapsto w_1+\cdots+w_p$ has
	density $b^{p-1}/(p-1)!$ with respect to $\dd b$.
	Hence the contribution of $\cR$
	to $Z(b)$ is
	$\q^{c(\cR)}e^{-b/2}b^{p-1}/(p-1)!$.
	Summing over all rigid quadrangulations therefore gives
	\[
	Z(b)=e^{-b/2}\sum_{p\ge1}A_p\frac{b^{p-1}}{(p-1)!}, \qquad A_p:=\sum_{\cR:\,\text{base has }p\text{ edges}}
	\q^{c(\cR)}\quad  \text{ for } p \geq 1.
	\]
	
	We now evaluate this sum. For $0 < x < 1$, we have 
	\[\q^2+\sum_{p\ge1}A_p x^p
	=\frac1{4x}
	-\frac{2-x}{4\pi x\sqrt{1-x}}
	\arccos\!\left(\frac{x}{2-x}\right) = \frac1{4\pi}\int_0^1
	\frac{s}{\sqrt{1-s^2}}\,
	\frac{\dd s}{1-\frac{x}{2}(1-s)},
	\]
	where the first equality follows from the bijection \cite[Theorem~5]{budd-flat-disks} and the generating-function formula in
	\cite[Section~9.3, p.~53]{bmep-eulerian-orientations}, and the second equality follows from integrating with $s=\cos\theta$. Extracting coefficients yields
	$A_p=\frac1{4\pi}\int_0^1
	\frac{s}{\sqrt{1-s^2}}
	\left(\frac{1-s}{2}\right)^p\dd s$, hence 
	\[
		Z(b)
		=e^{-b/2}\sum_{p\ge1}A_p\frac{b^{p-1}}{(p-1)!}
		=\frac{e^{-b/2}}{4\pi}\int_0^1
		\frac{s}{\sqrt{1-s^2}}\,
		\frac{1-s}{2}e^{b(1-s)/2}\dd s
		=\frac1{8\pi}\int_0^1
		s\sqrt{\frac{1-s}{1+s}}e^{-bs/2}\dd s.
	\]
	In particular, $Z(b)<\infty$. The substitution $s=2u/b$
 gives
	$Z(b)\sim(2\pi b^2)^{-1}$. Finally,
	\[
	\int_0^\infty Z(b)\dd b
	=\frac1{4\pi}\int_0^1\sqrt{\frac{1-s}{1+s}}\dd s
	=\frac1{4\pi}\left(\frac\pi2-1\right)
	=\frac18-\q^2. \qedhere
	\]
\end{proof}

\begin{remark}
	We gave the proof of Lemma~\ref{lem-partition} using the bijective result \cite[Theorem~5]{budd-flat-disks} mainly for simplicity of presentation. Alternatively, the partition function can be obtained
	directly from the recursive decomposition in the next section using methods from analytic combinatorics; see  \cite[Equation (25)]{budd-flat-disks} and \cite[Proposition 9.3]{bmep-eulerian-orientations}.
\end{remark}

\subsection{The row-by-row exploration}\label{subsec-row-by-row}

\begin{figure}[ht!]
	\begin{center}
		\includegraphics[scale=0.515]{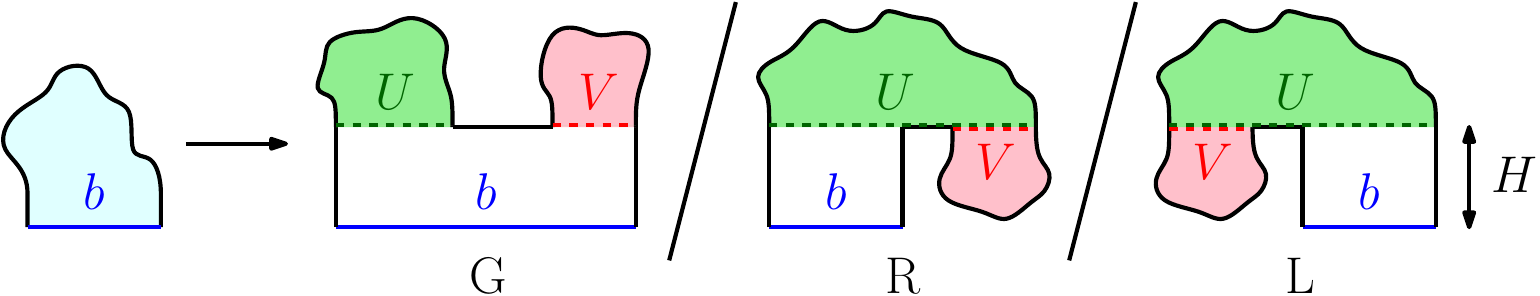}
	\end{center}
	\caption{\label{fig-row-decomposition} Branching events described in Proposition~\ref{prop-kernel}. Here, $H$ is the first-row height, and $U, V \geq 0$ are the base lengths of the children; base length zero corresponds to having an empty child.} 
\end{figure}

We now present the \emph{row-by-row exploration} introduced by Budd \cite{budd-flat-disks}. It gives a recursive self-similar decomposition of a Boltzmann
rectilinear disk; see Figure~\ref{fig-row-decomposition}.
Let $b>0$ and consider a sample $\cS$ from ${\P_b^{(1)}}$.
By definition, its base has length $b$ and is adjacent to two convex
corners. Orient $\cS$ so its base is horizontal and the disk is locally
above the base. From each concave corner of $\cS$, draw a horizontal
line segment until it hits the boundary of $\cS$ again; this cuts $\cS$
into finitely many smaller rectangles. Let $\Box$ be the
rectangle adjacent to the base, and let $H>0$ be its height.
The other pieces obtained by making the cuts described below are
called the \emph{children} of $\cS$. Each child is a rectilinear disk
whose base lies along the horizontal line containing the top edge
of $\Box$, and has the same (resp.\ opposite) orientation
as $\cS$ if it lies locally above (resp.\ below) its base.
There are three possible configurations:
\begin{itemize}
	\item (G event)
	The top edge of $\Box$ overlaps with some edge of $\cS$
	on an open interval. Cutting along the top edge of $\Box$
	gives at most two children with base lengths $U,V\geq0$,
	both with the same orientation as $\cS$.
	
	\item (R event)
	The right endpoint of the top edge of $\Box$ is a
	concave corner of $\cS$. Cut along the top edge of $\Box$
	extended to the right as far as possible. This gives at most
	two children with base lengths $U\geq b$ and $V\geq0$.
	The child with base length $U$ has the same orientation as $\cS$,
	whereas the child with base length $V$ has the opposite
	orientation.
	
	\item (L event)
	Same as R event but with ``right'' replaced by ``left''.
\end{itemize}
Note that any child with base length zero is empty; for instance,
in the G event with $U=V=0$ we simply have $\cS=\Box$.

Let $M(\dd u)=Z(u)\dd u+{\q^2}\delta_0(\dd u)$.
This is the natural measure describing the base length of Boltzmann
rectilinear disks when we also include a formal empty disk of width
zero, counted as having two corners and hence having weight ${\q^2}$. We have $|M| = 1/8$ by   Lemma~\ref{lem-partition}.
As we will see, the branching exploration is conveniently described
when using $M$.

\begin{proposition}\label{prop-kernel}
	Under ${\P_b^{(1)}}$, the height $H$ has law $\operatorname{Exp}(1)$ and is independent of everything below.
	The joint probability measure of each event and its child
	widths $(U,V)$ is
	\begin{align}
		\text{event }\mathrm G:\quad
		&\frac{e^{-(b-u-v)/2}}{Z(b)}
		\mathbf 1_{\{u+v<b\}}M(\dd u)M(\dd v),
		\label{eq-kernel-g}\\
		\text{each event }\mathrm R,\mathrm L:\quad
		&\frac{e^{-(u-b-v)/2}}{Z(b)}
		\mathbf 1_{\{u>b+v\}}Z(u)\dd u\,M(\dd v).
		\label{eq-kernel-r}
	\end{align}
	Conditional on the first-row type
	($\mathrm G$/$\mathrm R$/$\mathrm L$) and
	$(H,U,V)$, the nonempty children are independent
	with laws ${\P_U^{(1)}}$ and ${\P_V^{(1)}}$.
\end{proposition}

\begin{proof}
	The claim follows from two properties: factorization of the underlying Lebesgue measure $\lambda^b$, and factorization of the Boltzmann weight $w(\cS) = q_*^{c(\cS)} e^{-r(\cS)/2}$. 
	
	Firstly, on the restriction of $\lambda^b$ to $\mathrm G$, let $\cS_1$ and $\cS_2$ be the two children. Then the joint law of $(H, U, V, \cS_1, \cS_2)$ on the event that $\cS_1$ and $\cS_2$ are nonempty is 
	\[\mathbf 1_{\{u + v < b\}}
	\dd h\,\dd u\,\dd v\,
	\lambda^u(\dd\cS_1)\lambda^v(\dd\cS_2)\]
	where $\dd h$, $\dd u$, and $\dd v$ denote Lebesgue measure
	on $(0,\infty)$. Indeed, the row and column widths of $\cS$
	are precisely those of the two children, together with $h$
	and $x = b-u-v$ (after disintegrating at $b=u+v+x$, no
	$\dd x$ factor remains). For the case where $\cS_1$ is empty, we replace $\dd u\,  \lambda^u(\dd \cS_1)$ by $\delta_0(\dd u)\,  \delta_\emptyset (\dd \cS_1)$, and likewise for $\cS_2$. Thus, the restriction of ${\lambda^b}$ to $\mathrm G$ with no constraints on the children is $\mathbf 1_{\{u + v < b\}}
	\dd h\, (\dd u \, \lambda^u(\dd \cS_1) + \delta_0(\dd u) \, \delta_\emptyset (\dd \cS_1)) (\dd v \, \lambda^v(\dd \cS_2) + \delta_0(\dd v) \, \delta_\emptyset (\dd \cS_2))$.
	
	Now, by direct case-by-case analysis we have 
	\[
	c(\cS)=c(\cS_1) + c(\cS_2),
	\qquad
	r(\cS)=r(\cS_1) + r (\cS_2) + 2h + (b - u - v),
	\]
	where we set $c(\emptyset) = 2$ and $r(\emptyset) = 0$. Thus the Boltzmann weight factorizes as $w(\cS) = e^{-h} e^{-(b-u-v)/2} w(\cS_1) w(\cS_2)$. Combining this with the factorization of Lebesgue measure, we conclude that on the event $\mathrm G$ the law of $(H, U, V, \cS_1, \cS_2)$ is  \eqref{eq-kernel-g} times $e^{-h} \dd h \, \P^{(1)}_u (\dd \cS_1)\, \P^{(1)}_v (\dd \cS_2)$, where we set $\P^{(1)}_0 = \delta_\emptyset$. 
	
	One can argue similarly for the event $\mathrm R$: the Lebesgue measure restricted to $\mathrm R$ factorizes as $\mathbf 1_{\{u > b+v\}}
	\dd h\, \dd u \, \lambda^u(\dd \cS_1) (\dd v \, \lambda^v(\dd \cS_2) + \delta_0(\dd v) \, \delta_\emptyset (\dd \cS_2))$, whereas the weights factor as $w(\cS) = e^{-h} e^{-(u-b-v)/2} w(\cS_1) w(\cS_2)$, so the law of $(H, U, V, \cS_1, \cS_2)$ on $\mathrm R$ is  \eqref{eq-kernel-r} times $e^{-h} \dd h \, \P^{(1)}_u (\dd \cS_1)\, \P^{(1)}_v (\dd \cS_2)$. The same is true for $\mathrm L$. Thus the claim follows. 
\end{proof}

\subsection{The vertical exploration}
\label{subsec-discrete-vertical}

In this section we fix some scaling factor $B>0$. In Section~\ref{sec-discrete-convergence} $B$ will be sent to $\infty$ to take the scaling limit.

Consider a sample $\cS \sim \P_b^{(B)}$; as before, orient $\cS$ so its base is horizontal and it lies locally above its base. Pick a point from its base uniformly at random, and let $\eta$ be the vertical ray in $\cS$ from this point, parametrized by length, until it hits the boundary at time $T^{(B)}$. For rational $t \in (0, T^{(B)})$, let $L_t^{(B)}$ and $R_t^{(B)}$ be the lengths of the left and right horizontal rays from $\eta(t)$ to the boundary, and extend to all $t \geq 0$ by right-continuity and setting $L_t^{(B)} = R_t^{(B)} = 0$ for $t \geq T^{(B)}$. This gives the \emph{length process} $(L_t^{(B)}, R_t^{(B)})_{[0,\infty)}$ of the \emph{vertical exploration}. See  Figure~\ref{fig-exploration}.

At each $\text{G}/\text{R}/\text{L}$ event, the ray $\eta$ either stops (if it hits a horizontal boundary edge of $\cS$), or continues upward into a child having the same orientation as $\cS$. In this way, the vertical exploration discovers all children hit by $\eta$, leaving behind some unexplored children. 
By iterating Proposition~\ref{prop-kernel}, conditioned on $(L_t^{(B)}, R_t^{(B)})_{[0,\infty)}$ and on the base lengths of the unexplored children, the unexplored children are conditionally independent Boltzmann rectilinear disks whose conditional laws are $\P_u^{(B)}$ for their
respective base lengths $u>0$.

We now identify the transition kernel of the length process. Set
\[
k(u)=\int_{[0,u]}e^{-(u-v)/2}M(\dd v),\qquad u>0.
\]
Notice that $(L_t^{(B)}, R_t^{(B)})_{[0, \infty)}$ is a pure jump process, with jumps occurring during $\text{G}/\text{R}/\text{L}$ events.
By Proposition~\ref{prop-kernel}, the times between jumps are independent
$\operatorname{Exp}(B)$ times, and 
the  transitions are
\begin{equation}\label{eq-discrete-one-row}
	\begin{aligned}
		(\ell,r)\to(\ell+z,r)
		&:\quad
		B k(B|z|)
		\frac{Z(B(\ell+r+z))}{Z(B(\ell+r))}\dd z,
		&& z>-\ell,\\
		(\ell,r)\to(\ell,r+z)
		&:\quad
		B k(B|z|)
		\frac{Z(B(\ell+r+z))}{Z(B(\ell+r))}\dd z,
		&& z>-r,\\
		(\ell,r)\to(0,0)
		&:\quad
		\frac{k(B\ell)k(Br)}{Z(B(\ell+r))}.
	\end{aligned}
\end{equation}
In all three cases, the $k$ term comes from integrating out the unexplored child or children. 
Set
\[
h_B(\ell,r)=2\pi B^2Z(B(\ell+r)),\qquad \ell,r>0,
\]
and extend $h_B$ by zero outside $(0,\infty)^2$ and at
the cemetery state.
Let $P_t^{(B),+}(x,\dd y)$ be the transition kernel of
the symmetric compound Poisson process with L\'evy measure
$B^2k(B|z|)\dd z$, killed on leaving $(0,\infty)$.

\begin{lemma}\label{lem-discrete-h-transform}
	The length process $(L_t^{(B)},R_t^{(B)})$ has transition
	kernel on $(0,\infty)^2$
	\begin{equation}\label{eq-discrete-vertical-kernel}
		Q_t^{(B)}((\ell,r),\dd x\,\dd y)
		=
		\frac{h_B(x,y)}{h_B(\ell,r)}
		P_t^{(B),+}(\ell,\dd x)
		P_t^{(B),+}(r,\dd y).
	\end{equation}
	The missing mass is the probability of termination by time $t$.
	Thus the length process is the Doob $h_B$-transform of
	a pair of independent compound Poisson processes with the
	above L\'evy measure, killed on leaving $(0,\infty)^2$.
\end{lemma}

\begin{proof}
	Since $\int_0^\infty k(u)\dd u=2|M|=1/4$, the pair has
	total jump rate $B$.
	At each jump, its transition in either coordinate has density
	$Bk(B|z|)\dd z$.
	Multiplying each surviving transition by the ratio of $h_B$
	at the new and old states gives
	\eqref{eq-discrete-one-row}, with the missing mass accounting
	for termination.
	Both processes have independent $\operatorname{Exp}(B)$
	holding times, proving the claim.
\end{proof}

\section{JT sectors}\label{sec-jt-sectors}
The goal of this section is to introduce random flat surfaces we call JT sectors, and derive some of their fundamental properties using a conformal welding input from \cite{ang-jt-zipper}. 

In Section~\ref{subsec-flat-surfaces} we introduce the notion of flat surfaces and flow lines. In Section~\ref{subsec-jt-sectors} we recall the  angle-$\theta$ JT sector first introduced in  \cite{ang-jt-zipper} and state some properties proved in \cite{ang-jt-zipper}.  Section~\ref{subsec-rerooting} proves a rerooting symmetry of the angle-$\pi$ JT sector, and Sections~\ref{subsec-rerooting} and~\ref{subsec-quarter-partition-function} compute partition functions for the angle-$\pi$ and angle-$\pi/2$ JT sectors respectively. 

\subsection{Flat surfaces, developing maps, and flow lines}
\label{subsec-flat-surfaces}

A \emph{flat surface} is an oriented two-dimensional manifold
equipped with a smooth Riemannian metric of zero Gaussian curvature,
which is conformally equivalent to the disk.
Boundary points and arcs are understood via conformal uniformization.
An open boundary arc is \emph{straight} if it has a neighborhood
isometric to a subset of the closed upper half-plane, with the
arc mapped onto a bounded open interval of the real line.

	Parametrize a flat surface as
	$\cS=(\D_+,e^{2\sigma(z)}|dz|^2)$; flatness implies that $\sigma$ is harmonic. Since $\D_+$ is simply
	connected, $\sigma$ has a harmonic conjugate $\wt\sigma$. The holomorphic function
	$e^{\sigma+i\wt\sigma}$ admits a primitive $F$ satisfying
	\begin{equation}\label{eq-developing-map}
		F'(z)=e^{\sigma(z)+i\wt\sigma(z)},
		\qquad |F'(z)|=e^{\sigma(z)}.
	\end{equation}
	The map $F$ is an
	orientation-preserving local isometry from
	$(\D_+,e^{2\sigma(z)}|dz|^2)$ to $\C$. We call it a \emph{developing map} of $\cS$. 
	Note that $F$ is unique up to rotation and translation (these correspond to the extra freedom in choosing the harmonic conjugate and primitive respectively), and while $F$ is locally injective, it
need not be globally injective.

Fix a developing map $F$.
For a point $p\in\D_+$, the \emph{angle-$\theta$ flow line}
from $p$ is the maximal continuous curve
$\eta_p^\theta:[0,T_p^\theta)\to\D_+$ satisfying
\[
\eta_p^\theta(0)=p,\qquad
F(\eta_p^\theta(t))=F(p)+te^{i\theta},
\qquad 0\le t<T_p^\theta.
\]
This definition also holds if $p$ lies in the relative interior
of a straight boundary arc, and the direction $\theta$ points
into the surface. The \emph{length} or \emph{developed length}
of a flow-line segment $\gamma$ is the length of the
line-segment $F(\gamma)$. Finally, the \emph{horizontal chord}
through a point $p\in\D_+$ is the union of its angle-$0$ and
angle-$\pi$ flow lines.

\subsection{JT sectors and cutting}\label{subsec-jt-sectors}
In this section, we introduce the JT sector as in \cite{ang-jt-zipper}, and then state two inputs that we need from \cite{ang-jt-zipper}. 

Let $\mu$ be the law of the Gaussian holomorphic function
\begin{equation}\label{eq-gaussian}
	G(z)=\sum_{n\ge1}\frac{G_n}{\sqrt n}z^n,
	\qquad z\in\D,\qquad
	G_n\overset{\mathrm{i.i.d.}}{\sim}N(0,1).
\end{equation}
Almost surely, \eqref{eq-gaussian} converges on compact subsets of $\D$. 

\begin{definition}\label{def-jt-sector}
	For $\theta>0$, sample a pair $(G,A)$ from 
	\begin{equation}\label{eq-sector-measure}
		\Gamma(\theta/\pi)\,
		e^{(1-\theta/\pi)a}\mu(\dd g)\,\dd a ,
	\end{equation}
	and define the holomorphic function
	\begin{equation}\label{eq-primitive}
		F_{\theta,G,A}(z)
		=e^A\int_0^z
		\frac{w^{\theta/\pi-1}e^{G(w)}}{\sqrt{1-w^2}}\dd w,
		\qquad z\in\D_+.
	\end{equation}
	Here $0<\arg w<\pi$, and the square root is the branch equal
	to one at zero. 

	The flat surface with three marked points $(\D_+, |F_{\theta,G,A}'(z)|^2|dz|^2, 1, -1, 0)$ is called an \emph{angle-$\theta$ JT sector}. The point $0$ is called its \emph{vertex}. The boundary arcs $(-1,0)$ and $(0,1)$ are its left and right \emph{straight sides}, and the boundary arc $\partial \D_+ \backslash [-1,1]$ is its \emph{rough boundary}. We denote the law of the angle-$\theta$ JT sector by $\Smeas_\theta$. 
\end{definition}
In Definition~\ref{def-jt-sector}, 
the left and right straight sides have lengths $|F_{\theta,G,A}(-1)|$ and $F_{\theta,G,A}(1)$, and meet at its vertex with interior angle
$\theta$.

We note that the vertex of a JT sector a.s.\ determines the other two marked points, since they are a.s.\ the endpoints of the longest straight boundary arcs {emanating} from the vertex. When we discuss the flow lines of a JT sector, we use the developing map which sends its vertex to $0$ and its right straight side to a subset of $(0, \infty)$. 

In this paper, we only work with angle-$\theta$ JT sectors for $\theta \in \{ \pi/2, \pi\}$, and typically use $\cS$ to denote an angle-$\pi$ JT sector and $\cQ$ for an angle-$\pi/2$ JT sector. 

We now state two inputs that we need from \cite{ang-jt-zipper}. 
\begin{proposition}[{\cite[Proposition~A.1]{ang-jt-zipper}}]\label{prop-boundary-continuity}
	For $\theta>0$ and $\mu$-almost every $g$,
	the map $F_{\theta,g,0}$ in \eqref{eq-primitive} extends continuously
	to $\ol{\D_+}$.
\end{proposition}
The proof of Proposition~\ref{prop-boundary-continuity} was generated using GPT-6 Astra (Ultra). It draws
heavily on the occupation-time approach of A\"id\'ekon, Hu, and
Shi~\cite{ahs-discounted-tree-sums}, adapting their branching random walk arguments to our continuum setting. We claim no credit for this argument.  

As a consequence of Proposition~\ref{prop-boundary-continuity}, the lengths of the left and right straight sides of a JT sector are a.s.\ finite. 
Let $\{\Smeas_\theta (\ell, r)\}_{\ell, r > 0}$ be the disintegration of $\Smeas_\theta$ with respect to left and right  side lengths. That is,  $\Smeas_\theta = \iint_0^\infty \Smeas_\theta(\ell, r)  \dd \ell \dd r$, and each $\Smeas_\theta(\ell, r)$ is supported on the space of flat surfaces having  side lengths $(\ell, r)$. We define the partition functions\footnote{\cite[Chapter~4]{van-den-dungen-flat-disks} computed
	scaling limits of various discrete flat disk-sector partition functions,
	obtaining explicit expressions for
	their Mellin transforms.} 
\begin{equation}\label{eq-sector-partition-function}
	Z_\theta(\ell,r):=|\Smeas_\theta(\ell,r)|.
\end{equation}
These are determined
in Sections~\ref{subsec-rerooting} and~\ref{subsec-quarter-partition-function} for $\theta \in \{\pi/2, \pi\}$.

\begin{figure}[ht!]
	\begin{center}
		\includegraphics[scale=0.55]{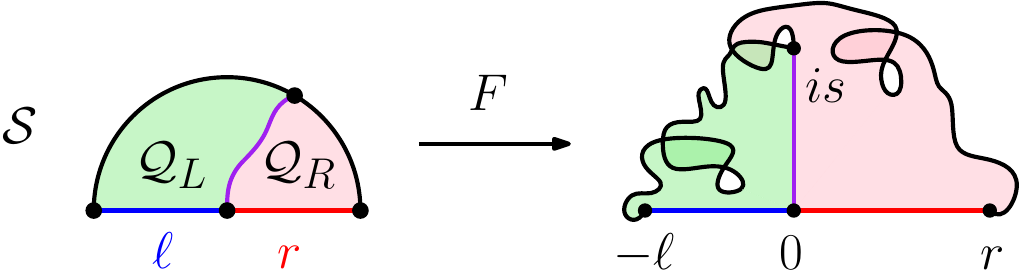}
	\end{center}
	\caption{\label{fig-weld} The cutting identity of Theorem~\ref{thm-sector-cutting}, shown in conformal coordinates (left) and under the developing map \(F\) (right). The vertical flow line (purple) has random length $s$ and is mapped to $[0,is]$. Cutting along it separates the angle-$\pi$ JT sector $\cS$ into two angle-$\pi/2$ sectors $\cQ_L,\cQ_R$, with respective side lengths $(\ell,s)$ and $(s,r)$.} 
\end{figure}

The key input underlying our argument is the following decomposition result, which says that cutting an angle-$\pi$ JT sector by a vertical flow line decomposes it into two angle-$\pi/2$ JT sectors having a shared side length. See Figure~\ref{fig-weld}. 

\begin{theorem}[{\cite[Theorem~1.1]{ang-jt-zipper}}]\label{thm-sector-cutting}
	The following holds for a.e.\ $\ell, r > 0$. Sample a flat surface $\cS$ from $\Smeas_\pi(\ell, r)$. Let $\cQ_L, \cQ_R$ be the left and right connected components obtained when $\cS$ is cut along the angle-$\pi/2$ flow line from its vertex. Then the joint law of $\cQ_L$ and $\cQ_R$ is 
	\begin{equation}\label{eq-cutting}
		\int_0^\infty
		\Smeas_{\pi/2}(\ell,s)\otimes
		\Smeas_{\pi/2}(s,r)\dd s.
	\end{equation}
\end{theorem}

\subsection{The angle-\texorpdfstring{$\pi$}{pi} JT sector rerooting symmetry and partition function}
\label{subsec-rerooting}

In this section we give a description of the angle-$\pi$ JT sector with lengths $(\ell, r)$ for all $\ell, r > 0$ (rather than just a.e.\ $\ell, r>0$), prove a re-rooting symmetry, and compute the partition function $Z_\pi$. 

\begin{definition}\label{def-conditioned-sector-law}
	For $b>0$, sample $G$ as in \eqref{eq-gaussian} and set $A = \log b - \log (\int_{-1}^1
	\frac{e^{G(x)}}{\sqrt{1-x^2}}\dd x)$.
	Let $\P_b$ be the law of the flat surface
	\[
	\bigl(\D_+,|F_{\pi,G,A}'(z)|^2|dz|^2\bigr),
	\]
	with $[-1,1]$ distinguished as its base and the original vertex
	$0$ forgotten; by construction its base length is $b$.
	For $\ell,r>0$, let $\P_{\ell,r}$ be the law obtained by
	sampling a surface from $\P_{\ell+r}$ and marking the point
	at base arclength $\ell$ from the left endpoint as its vertex.
\end{definition}

A simple change of variables relates the angle-$\pi$ JT sector with vertex forgotten to $\P_b$. 
\begin{lemma}\label{lem-unpointed-sector}
	Consider the function which, given a flat surface with marked point, forgets the point. 
	The pushforward of $\Smeas_\pi$ under this function is  $\int_0^\infty b^{-1}\P_b\dd b$. 
\end{lemma}

\begin{proof}
$\Smeas_\pi$ is defined by sampling
$(G,A)$ from $\mu(\dd g)\,\dd a$. The base length is 
\[
B
=e^A\int_{-1}^1
\frac{e^{G(u)}}{\sqrt{1-u^2}}\dd u.
\]
For each fixed $g$, changing variables from $a$ to $b$
gives $\dd a=\dd b/b$, so the joint law of $(G,B)$ is $\mu(\dd g) \, b^{-1} \dd b$.
Fixing $B=b$ amounts to sampling $G\sim\mu$ and choosing $A$
so that the base length is $b$, exactly as in
Definition~\ref{def-conditioned-sector-law}. Thus, for a sample from $\Smeas_\pi$, after forgetting the vertex,
the law of the resulting surface is 
$\int_0^\infty b^{-1} \P_b\dd b$. 
\end{proof}

The next result establishes that the vertex of an angle-$\pi$ JT sector is distributed uniformly by length measure on its base.  The proof is inspired by that of \cite[Proposition 2.18]{AHS-SLE-integrability}.

\begin{proposition}\label{prop-reroot}
	We have 
	\begin{equation}\label{eq-half-sector-partition-function}
		\Smeas_\pi(\ell,r)=\frac{1}{(\ell+r)^2}\P_{\ell, r}
		\qquad\text{for Lebesgue-almost every }\ell,r>0.
	\end{equation}
	In particular, the angle-$\pi$ JT  sector is invariant in law under resampling the root from the length measure on its base, and the partition function \eqref{eq-sector-partition-function} satisfies $Z_\pi(\ell, r) = (\ell + r)^{-2}$. 
\end{proposition}

\begin{proof}
	\noindent\textbf{Step 1: Parametrization in $\mathbb S$.} We work in the strip where computations are easier due to translation symmetry. 
	Set 
	\[
	\mathbb S=\{z\in\C:0<\Im z<\pi/4\},
	\qquad
	X(z)=G(\tanh z),
	\qquad
	q(z)=\log\cosh z,
	\]
	where the branch of $\log\cosh z$ is chosen to be real on $\R$.
	Let $\nu$ be the law of the centered Gaussian field $X$. The covariance can be computed using that of $G$, but for our purposes all we need is that for $z\in\mathbb S$ and $t\in\R$,
	\[
	\E[X(z)X(t)]
	=q(z)+q(t)-q(z-t),
	\qquad
	\E[X(t)^2]=2q(t).
	\]
	Using this covariance (or by conformal invariance of the Gaussian free field), it is easy to check that 
	\begin{equation}\label{eq-strip-covariance}
		X(\cdot+t)-X(t)\overset{d}=X
		\qquad\text{for every }t\in\R.
	\end{equation}
	
	Given a pair $(X, A)$, we can define a field \[
			\Phi(z)=X(z)-q(z)+A.\]
	Let $\cM$ and $\cN$ be the laws of $\Phi$
	when $(X,A)$ is sampled from
	$\nu(\dd x)\,\dd a$
and 
	$\nu(\dd x)\,e^{-a}\dd a$ 
	respectively.
	
	The map $z\mapsto\tanh z$ conformally maps $\mathbb S$
	onto $\D_+$, and
	\[
	F(z):=\int_0^z e^{\Phi(w)}\dd w
	=F_{\pi,G,A}(\tanh z).
	\]
	Therefore, if $\Phi$ is sampled from $\cM$, then
	the flat surface
	$\bigl(\mathbb S,e^{2\Re\Phi(z)}|\dd z|^2,0\bigr)$  has law $\Smeas_\pi$.

	\smallskip
	\noindent\textbf{Step 2: Translation invariance of $\cN$.}
	Fix $t\in\R$ and write $T_t\Phi=\Phi(\cdot+t)$.
	We will prove that
	\begin{equation}\label{eq-strip-translation}
		(T_t)_*\cN=\cN.
	\end{equation}
	Given $(X, A)$ and $\Phi(z) = X(z) - q(z) + A$, we have $\Phi(z+t) = Y(z) - q(z) + C$, where 
	\[
		Y(z)=X(z+t)-X(t)-q(z+t)+q(z)+q(t),\qquad 
		C=A+X(t)-q(t).
	\]
	We must show that if the law of $(X,A)$ is 
	$\nu(\dd x)\,e^{-a}\dd a$, then the law of $(Y,C)$ is 
	$\nu(\dd y)\,e^{-c}\dd c$.
	
	For fixed $x$, changing variables from $a$ to
	$c=a+x(t)-q(t)$ gives the measure
	\[ \nu_t (\dd x) \, e^{-c} \dd c, \qquad \text{ where } \nu_t (\dd x) := 
	e^{x(t)-q(t)}\nu(\dd x).
	\]
	We now use Girsanov's theorem. We have $\E [X(z) X(t)] = q(z)+q(t)-q(z-t)$. Thus, the $\nu_t$-law of $X$ is the law of $X(z) + q(z) +q(t)-q(z-t)$ where $X \sim \nu$. Substituting this into the formula for $Y$, the $\nu_t$-law of $Y$ is the law of 
	\[(X(z+t) + q(z+t) + q(t) - q((z+t)-t)) - (X(t) +q(t) + q(t) - q(0))   - q(z+t) + q(z) + q(t)\]
	where $X \sim \nu$. Cancelling terms, this large expression simplifies to $X(z + t) - X(t)$,  which has law $\nu$ by \eqref{eq-strip-covariance}. We conclude that the $\nu_t$-law of $Y$ is $\nu$. This completes the proof of \eqref{eq-strip-translation}. 
	
	\smallskip
	\noindent\textbf{Step 3: Rerooting symmetry of $\Smeas_\pi$.}
	Since $X(0)=q(0)=0$, we have
	$\cM(\dd\phi)
	=e^{\phi(0)}\cN(\dd\phi)$.
	Write 
	$L(\Phi)=\int_\R e^{\Phi(t)}\dd t$ 
	for the total base length, so $L(\Phi) \in (0,\infty)$.
	For every nonnegative measurable functional $f$,
	translation invariance of $\cN$ gives
	\begin{align*}
		\int 
		\int_\R 	\frac{e^{\phi(t)}}{L(\phi)}f(T_t\phi)\dd t \, &\cM(\dd\phi)
		=
		\int_\R\int 
		\frac{e^{\phi(0)+\phi(t)}}{L(\phi)}
		f(T_t\phi)\, \cN(\dd\phi)\dd t
		=
		\int_\R\int 
		\frac{e^{\phi(-t)+\phi(0)}}{L(\phi)}
		f(\phi)\, \cN(\dd\phi)\dd t \\
		&=
				\int\int_\R 
		\frac{e^{\phi(-t)}}{L(\phi)}
		\dd t \, f(\phi) \cM(\dd\phi)
		=
		\int f(\phi)\cM(\dd\phi).
	\end{align*}
	The identity $\int 
	\int_\R 	\frac{e^{\phi(t)}}{L(\phi)}f(T_t\phi)\dd t \, \cM(\dd\phi) = \int f(\phi)\cM(\dd\phi)$ says that if you sample $\Phi \sim \cM$ and reroot at a point chosen uniformly from the base length measure, the law of the resulting field is still $\cM$. This means $\Smeas_\pi$ is invariant in law under resampling the vertex
	from the base length measure.
	
	\smallskip 
	\noindent \textbf{Step 4: Conclusion.}
By Lemma~\ref{lem-unpointed-sector}, for a sample from $\Smeas_\pi$ with vertex forgotten,
the law of the resulting surface is 
$\int_0^\infty b^{-1} \P_b\dd b$. By Step 3, we can restore the vertex by sampling it uniformly from base length measure, hence 
	\[
	\Smeas_\pi
	=\int_0^\infty\int_0^b
	\P_{\ell,b-\ell}\,
	\frac{\dd\ell\,\dd b}{b^2}
	=\int_0^\infty\int_0^\infty
	\frac{\P_{\ell,r}}{(\ell+r)^2}
	\,\dd\ell\,\dd r. \qedhere
	\]
\end{proof}

\subsection{The angle-\texorpdfstring{$\pi/2$}{pi/2} partition function}
\label{subsec-quarter-partition-function}
In Proposition~\ref{prop-quarter-identification} we identify $Z_{\pi/2}$ with the entrance density of a Cauchy excursion. We state the consequences we need downstream in Corollary~\ref{cor-quarter-partition-function}. 

Let $p_t^+(x,y)$ be the transition density of the standard symmetric
Cauchy process killed on leaving
$(0,\infty)$. Let $e_t(x)$ be its excursion entrance density,
normalized such that
\begin{equation}\label{eq-entrance-mass}
	\int_0^\infty e_t(x)\dd x=(\pi t)^{-1/2}.
\end{equation}

\begin{proposition}
	\label{prop-quarter-identification}
For all $\ell, r > 0$ we have 
	$Z_{\pi/2}(\ell,r)=\sqrt\pi\,e_r(\ell)$.
\end{proposition}
Essentially en route to showing Proposition~\ref{prop-quarter-identification}, we will obtain the following. 

\begin{corollary}\label{cor-quarter-partition-function}
We have $	\int_0^\infty Z_{\pi/2}(\ell,r)\dd\ell =r^{-1/2}$ and 
\begin{equation}\label{eq-quarter-identities}
	\int_0^\infty Z_{\pi/2}(\ell,s)Z_{\pi/2}(r,s)\dd s
	=\frac1{(\ell+r)^2},
	\qquad
	\int_0^\infty p_t^+(x,y)Z_{\pi/2}(y,r)\dd y
	=Z_{\pi/2}(x,t+r).
\end{equation}
\end{corollary}
To prove Proposition~\ref{prop-quarter-identification}, we will show that $Z_{\pi/2}$ is characterized by several properties (Lemma~\ref{lem-mellin}), and that $\sqrt\pi\,e_r(\ell)$ also satisfies these properties (Lemma~\ref{lem-cauchy-gram}). 

\begin{lemma}\label{lem-mellin}
	Up to equality almost everywhere, there is at most one nonnegative
	measurable function $N$ on $(0,\infty)^2$ satisfying
	\[
	N(\ell,r)=N(r,\ell),\qquad
	N(c\ell,cr)=c^{-3/2}N(\ell,r),\qquad
	\int_0^\infty N(\ell,s)N(r,s)\dd s
	=\frac1{(\ell+r)^2}
	\]
	for every $c>0$ and almost every
	$(\ell,r)\in(0,\infty)^2$.
\end{lemma}
\begin{proof}
	Write $N(\ell, r) = \ell^{-3/2}n(r/\ell)$, and let
		$\cM(z)=\int_0^\infty u^{z-1}n(u)\dd u$
		be the Mellin transform of $n$.
		
		We have 
		$N(1,s)N(r,s)=n(s)s^{-3/2}n(r/s)$, so $(1+r)^{-2} = \int_0^\infty n(s) s^{-3/2} n(r/s) \dd s$. For $z \in (0,2)$, we multiply by $r^{z-1}$ and integrate over $r$:
\[
			\int_0^\infty \frac{r^{z-1}}{(1+r)^2}\dd r
			=\int_0^\infty n(s)s^{-3/2}
			\left(\int_0^\infty r^{z-1}n(r/s)\dd r\right)\dd s
			=\cM(z)\int_0^\infty s^{z-3/2}n(s)\dd s
			=\cM(z)\cM(z-\tfrac12).
\]
In the second equality, we used the substitution
$v=r/s$. 
		Evaluating the left-hand side by the beta integral yields
		\begin{equation}\label{eq-mellin}
			\cM(z)\cM(z-\tfrac12)
			=\Gamma(z)\Gamma(2-z),
			\qquad 0<z<2.
		\end{equation}
		In particular, $\cM(z) = \int_0^\infty u^{z-1}n(u)\dd u$ is convergent for $-1/2 < z  <2$,
		and hence absolutely convergent for $z \in \C$ with $-1/2  <\mathrm{Re}\, z < 2$.  Thus $\cM$ is holomorphic on this strip, with $|\cM (z)| \leq \cM(\mathrm{Re}\, z)$. By analytic continuation \eqref{eq-mellin} holds for $0 < \mathrm{Re} \, z < 2$. Note that $\cM$ has no zeros on $0  <\mathrm{Re} \, z < 2$ since $\Gamma(z) \Gamma(2-z)$ has no zeros there and $\cM(z-1/2)$ has no poles there. 
		
		We now show that \eqref{eq-mellin} uniquely identifies $\cM$. Consider two functions $\cM_1, \cM_2$ satisfying \eqref{eq-mellin}, and let   $R(z)=\cM_1(z)/\cM_2(z)$, so $R$ is holomorphic on $0 < \mathrm{Re}\, z < 2$.
		Equation \eqref{eq-mellin} gives
		$R(z)R(z-\tfrac12)=1$ and hence $R(z+1)=R(z)$ for $0 < \mathrm{Re}\, z < 1$. Thus we can holomorphically extend $R$ to $\C$, where it is nonvanishing.  
		
Now, consider $1/2 \leq x \leq 3/2$. We have 
\[R(x + iy) = \frac{\cM_1(x+iy) }{ \cM_2(x+iy)} = \frac{\cM_1(x+iy) \cM_2(x- \frac12+iy) }{\Gamma(x+iy)(\Gamma(2-x-iy))},\]
so using 
$|\cM_j(x+iy)|\le\cM_j(x)$ and $\Gamma(z)\Gamma(2-z)=\pi(1-z)/\sin(\pi z)$ gives 
\[
|R(x+iy)|
\le
\frac{\cM_1(x)\cM_2(x-\tfrac12)}
{|\Gamma(x+iy)\Gamma(2-x-iy)|}
\le Ce^{\pi|y|}.
\]
By periodicity, this bound extends from $x \in [1/2,3/2]$ to $\C$. The Hadamard factorization theorem thus implies that $R(z) = e^{az+b}$ for some $a,b \in \C$. Since $R(z) R(z-\tfrac12) = 1$ we conclude $R$ is constant, with $|R| = 1$. Positivity on $(0,2)$ implies $R \equiv 1$. 

Therefore \eqref{eq-mellin} characterizes $\cM$. Uniqueness of Mellin transforms then yields the claim. 
\end{proof}

\begin{lemma}\label{lem-cauchy-gram}
	The entrance density $e_t(x)$ satisfies
	\[
	e_t(x)=e_x(t),
	\qquad
	\int_0^\infty e_\ell(s)e_r(s)\dd s
	=\frac1{\pi(\ell+r)^2},
	\qquad
	\int_0^\infty p_t^+(x,y)e_s(y)\dd y=e_{t+s}(x)
	\]
	for all positive arguments.
\end{lemma}

\begin{proof}
		Using \cite[Chapter IV, (7), and Theorem VI.5(ii)]{bertoin-levy}
	and the normalization \eqref{eq-entrance-mass}, we compute
\[
		\int_0^\infty\int_0^\infty
		e^{-\alpha t-\beta x}e_t(x)\dd x\,\dd t
		=\exp\left(
		-\frac1\pi\int_0^\infty\int_0^\infty
		\frac{e^{-t}-e^{-\alpha t-\beta x}}{t^2+x^2}
		\dd x\,\dd t
		\right) \quad \text{ for }\alpha, \beta > 0.
\]
	This expression is symmetric in $\alpha,\beta$, implying 
	$e_t(x)=e_x(t)$.

	For the second identity, decomposing a Cauchy excursion of duration $\ell + r$ at time $\ell$, the partition function is $\int_0^\infty e_\ell(s)e_r(s)\dd s$. Decomposing it instead at its midpoint shows that this partition function equals $\int_0^\infty e_{(\ell+r)/2}(s)e_{(\ell+r)/2}(s)\dd s$, which is proportional to $(\ell + r)^{-2}$ by Cauchy scaling ($e_t(x)=t^{-3/2}e_1(x/t)$). We conclude that 
	\[\int_0^\infty e_\ell(s)e_r(s)\dd s = \frac{C}{(\ell+r)^2} \qquad \text{ for some } C>0.\]
	To pin down $C$, we use symmetry and integrate out $r$:
	\[ \frac C\ell = \int_0^\infty \frac{C}{(\ell+ r)^2} \dd r = \int_0^\infty e_s(\ell)
	\left(\int_0^\infty e_s(r)\dd r\right)\dd s
	=\frac1{\sqrt\pi}\int_0^\infty
	s^{-2}e_1(\ell/s)\dd s
	=\frac1{\sqrt\pi\,\ell}\int_0^\infty e_1(u)\dd u
	=\frac1{\pi\ell}.\]
	Finally, the third identity is simply the entrance law relation. 
\end{proof}

\begin{proof}[Proof of Proposition~\ref{prop-quarter-identification}]
	We claim that $Z_{\pi/2}$ satisfies the three properties of Lemma~\ref{lem-mellin}. Indeed, the first follows from reflection symmetry, the second from the fact that scaling by $c$ corresponds to adding $\log c$ to the zero mode $a$ in \eqref{eq-sector-measure}, and the third from Theorem~\ref{thm-sector-cutting} and 
	Proposition~\ref{prop-reroot}.

	By Cauchy scaling and Lemma~\ref{lem-cauchy-gram}, the function
	$(\ell,r)\mapsto\sqrt\pi\,e_r(\ell)$ also satisfies these three properties.
	Lemma~\ref{lem-mellin} therefore gives $Z_{\pi/2} (\ell, r) = \sqrt\pi e_r(\ell)$ as needed. 
\end{proof}

\begin{proof}[Proof of Corollary~\ref{cor-quarter-partition-function}]
	Combine 
	Proposition~\ref{prop-quarter-identification} with 
	\eqref{eq-entrance-mass} and 
	Lemma~\ref{lem-cauchy-gram}.
\end{proof}

\section{Vertical exploration of a JT sector}
\label{sec-vertical-exploration}

In this section we introduce the vertical exploration of an angle-$\pi$ JT sector, derive the law of its length process, and prove a strong Markov property for the unexplored sector at stopping times. 

\subsection{Exploring an angle-\texorpdfstring{$\pi/2$}{pi/2} JT sector}\label{subsec-quarter-exploration}
Let $x,t,r>0$.  Sample $\cQ_0$ from the probability measure $\Smeas_{\pi/2} (x, t+r) / Z_{\pi/2} (x, t+r)$; from the boundary point at distance $t$ counterclockwise of the vertex, cut along the angle-$\pi/2$ flow line to get two flat surfaces $\cA$ and $\cQ$, and let $Y$ be the length of the cut.  See Figure~\ref{fig-arch} (left).

\begin{figure}[ht!]
	\begin{center}
		\includegraphics[scale=0.49]{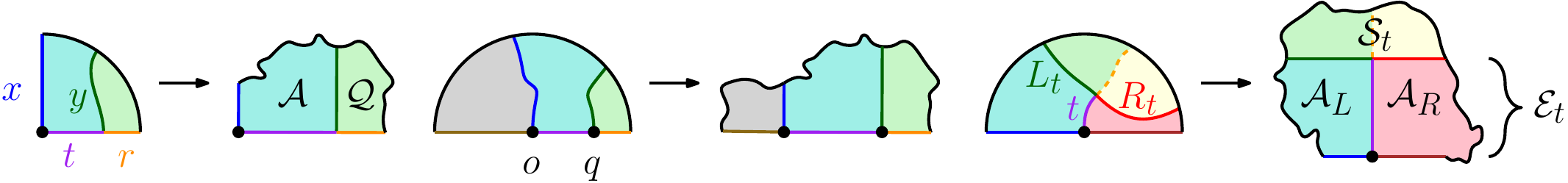}
	\end{center}
	\caption{\label{fig-arch} 
		\textbf{Left:} When cutting an angle-$\pi/2$ JT sector by a vertical flow line, Proposition~\ref{prop-arch} gives the law of the cut length and a stationarity property: conditioned on $\cA$, the law of $\cQ$ is $\Smeas_{\pi/2}$ with specified boundary lengths. \textbf{Middle:}  The proof of Proposition~\ref{prop-arch} starts from an angle-$\pi$ JT sector, marks a second point, and cuts by vertical flow lines from both points.  \textbf{Right:} Setup for Proposition~\ref{prop-retained} and its proof. The horizontal chord through $p_t$ cuts the angle-$\pi$ JT sector into an explored region $\cE_t$ and a to-be-explored region $\cS_t$, and has lengths $(L_t,R_t)$. Cutting $\cE_t$ along the vertical ray gives $\cA_L,\cA_R$.} 
\end{figure}

\begin{proposition}\label{prop-arch}
	For each fixed $t>0$ and a.e.\ $x, r > 0$, the law of the cut length $Y$ is 
	\begin{equation}\label{eq-quarter-cut-length}
		\P(Y\in\dd y)
		=\frac{p_t^+(x,y)Z_{\pi/2}(y,r)}
		{Z_{\pi/2}(x,t+r)}\dd y.
	\end{equation}
	Conditioned on $Y = y$, the flat surfaces $\cA$ and $\cQ$ are independent, and the law of $\cQ$ is $\Smeas_{\pi/2}(y, r)/Z_{\pi/2} (y,r)$. The conditional law of $\cA$ depends only on $(x,t,y)$. 
\end{proposition}

\begin{proof}
		Fix $t>0$, and write $y(\cA)$ for the cut length. Let $\wt {\Ameas}_t (x, \dd \cA)$ denote the law of $\cA$ when $\cQ_0$ is instead sampled from $	\int_t^\infty \Smeas_{\pi/2}(x,v)\dd v$; this measure is finite by Corollary~\ref{cor-quarter-partition-function}.	Set
		$\Ameas_t(x,\dd\cA)
		=\sqrt{y(\cA)}\,\wt {\Ameas}_t (x, \dd \cA)$.
		
		\smallskip 
		\noindent \textbf{Step 1: The cutting identity.}
		See Figure~\ref{fig-arch}  (middle). 
		Sample $\cS \sim \Smeas_\pi$ restricted to the event that its right side length exceeds $t$. Let $o$ be its vertex and $q$ the point
		at length $t$ to its right. 
		Cut at $o$ along a vertical flow line, and call the cut length $X$ and the flat surface on the right $\cQ_0$. By Theorem~\ref{thm-sector-cutting} and Corollary~\ref{cor-quarter-partition-function}, the law of $(X, \cQ_0)$ is 
		\[
		x^{-1/2}\dd x
		\int_t^\infty \Smeas_{\pi/2}(x,v)\dd v.
		\]
		Alternatively, by  Proposition~\ref{prop-reroot} the flat surface $\cS$ with its vertex replaced by $q$ is precisely a sample from $\Smeas_\pi$ restricted to the event that its \emph{left} side length exceeds $t$. 
		Cutting at $q$, by Theorem~\ref{thm-sector-cutting} the conditional law of the right
		component $\cQ$ given the left component is
		\[
		\sqrt y \int_0^\infty \Smeas_{\pi/2}(y,r)\dd r, 
		\]
		where $y$ is the cut length. Note that this left component determines $\cA$. 
		
		Combining these two facts and disintegrating in $(x,r)$, we see that for almost every $x,r>0$, cutting a sample of $\Smeas_{\pi/2}(x, t+r)$ gives a pair $(\cA, \cQ)$ with law 
		\begin{equation}\label{eq-arch-completion}
			\Ameas_t(x,\dd\cA)\,
			\Smeas_{\pi/2}(y(\cA),r;\dd\cQ).
		\end{equation}
		
		\smallskip 
		\noindent \textbf{Step 2: Extracting the law of the cut length.}
		Let $K_t(x,\dd y)$ be the pushforward of
		$\Ameas_t(x,\dd\cA)$ under
		$\cA\mapsto y(\cA)$. For a.e.\ $x,r>0$ we have 
		\[
		\int_0^\infty K_t(x,\dd y)Z_{\pi/2}(y,r) = Z_{\pi/2}(x,t+r) = \int_0^\infty p_t^+(x,y)Z_{\pi/2}(y,r)\dd y
		\]
		where the first equality holds by  \eqref{eq-arch-completion} and the second by Corollary~\ref{cor-quarter-partition-function}.  Fix $x$ for which these identities hold for almost every $r$. Let $b > 0$; multiply both sides by $Z_{\pi/2}(r,b)$ and integrate over $r>0$. By  Corollary~\ref{cor-quarter-partition-function}, 
		\[
		\int_0^\infty\frac{K_t(x,\dd y)}{(y+b)^2}
		=\int_0^\infty\frac{p_t^+(x,y)}{(y+b)^2}\dd y
		<\infty.
		\]
		Differentiating $j\geq 0$ times at $b=1$ gives
		$\int_0^\infty (1+y)^{-j-2}K_t(x,\dd y)
		=
		\int_0^\infty (1+y)^{-j-2}p_t^+(x,y)\dd y$. Thus the finite measures $(1+y)^{-2}K_t(x,\dd y)$ and
		$(1+y)^{-2}p_t^+(x,y)\dd y$ agree on all polynomials
		in $(1+y)^{-1}$, and hence 
		$K_t(x,\dd y)=p_t^+(x,y)\dd y$.
		
		Disintegrating \eqref{eq-arch-completion} over $y(\cA)$
		and dividing by $Z_{\pi/2}(x,t+r)$ gives
		\eqref{eq-quarter-cut-length}, and the remaining claims also follow from \eqref{eq-arch-completion}. 
\end{proof}

\subsection{Length process and Markov property}
\label{subsec-vertical-markov}
In this section, we introduce the vertical exploration, and deduce length laws and a Markov property at dyadic times. 

Let $\ell, r > 0$. Consider a sample $\cS_0$ from $\P_{\ell, r}$, and call its vertex $p_0$. Recall from Section~\ref{subsec-flat-surfaces} the notions of flow line and horizontal chord. 
For the flow line $\eta_{p_0}^{\pi/2}$, let $T$ be its duration and write $p_t = \eta_{p_0}^{\pi/2}(t)$ for $0 < t < T$. 
For each dyadic $t \in (0,T)$, cut $\cS_0$ by the horizontal chord through $p_t$ 
 to get two flat surfaces $\cS_t$ and $\cE_t$. Here, $\cE_t$ lies below the cut and has a marked point $p_0$, whereas $\cS_t$ lies above the cut and comes with a marked point $p_t$. Let $L_t$ and $R_t$ be the side lengths of $\cS_t$. 
 {For dyadic $s\ge T$, set $\cE_s=\dagger$,
where $\dagger$ is a fixed cemetery state.}
 Define the filtration $\cF_t^0=\sigma(\cE_s:0<s\le t,\ s\text{ dyadic})$. Note in particular that $L_t$ and $R_t$ are $\cF_t^0$-measurable, since these lengths can be obtained by following the vertical flow line from $p_0$.

We combine Theorem~\ref{thm-sector-cutting} and Proposition~\ref{prop-arch} to prove the following. Essentially, the length process $(L_t, R_t)$ has the law of a pair of independent Cauchy processes killed on leaving the positive quadrant, under a Doob $h$-transform forcing them to die continuously at $(0,0)$ in finite time. 

\begin{proposition}\label{prop-retained}
	For a.e.\ $\ell, r>0$, the following holds under $\P_{\ell,r}$ at every dyadic time $t>0$.
	On $\{t<T\}$, the conditional law of $\cS_t$ given $\cF_t^0$ is $\P_{L_t,R_t}$.
	Moreover, $(L_t,R_t)$ is Markov with respect to
	$(\cF_t^0)$, with  transition kernel
	\begin{equation}\label{eq-tagged-kernel}
		Q_t((\ell,r),\dd x\,\dd y)
		=\frac{Z_\pi(x,y)}{Z_\pi(\ell,r)}
		p_t^+(\ell,x)p_t^+(r,y)\dd x\,\dd y.
	\end{equation}
	The missing mass in \eqref{eq-tagged-kernel} is the probability of reaching $(0,0)$ by time $t$. 
\end{proposition}

\begin{proof}
		Fix a dyadic $t>0$. On $\{t<T\}$, let
		$\cA_L,\cA_R$ be the two flat surfaces obtained by
		cutting $\cE_t$ along the vertical flow line from
		$p_0$ to $p_t$.
		By Theorem~\ref{thm-sector-cutting}, cutting $\cS_0$ along the vertical
		flow line gives two angle-$\pi/2$ JT sectors $\cQ_0^L$ and $\cQ_0^R$ with common side
		length $T$. Let $\cQ_t^L$ and $\cQ_t^R$ be the two flat surfaces obtained by cutting $\cS_t$ along its vertical flow line. 
		Apply Proposition~\ref{prop-arch} to each of $\cQ_0^L$ and $\cQ_0^R$.
		Conditional on $(L_t,R_t)=(x,y)$ and $s = T-t$, the four pieces $\cA_L,\cA_R, \cQ_t^L, \cQ_t^R$ are conditionally independent; the conditional laws of $\cA_L$ and $\cA_R$  depend only on $(\ell,t,x)$ and
		$(r,t,y)$, and the conditional laws of $\cQ_t^L$ and $\cQ_t^R$ are $\Smeas_{\pi/2}(x, s)/Z_{\pi/2}(x,s)$ and $\Smeas_{\pi/2}(s, y) / Z_{\pi/2}(s,y)$. Integrating over $s$ and gluing $\cQ_t^L$ and $\cQ_t^R$ using \eqref{eq-cutting}, we obtain
		\begin{equation}\label{eq-surface-markov}
			\P_{\ell,r}\bigl(
			L_t\in\dd x,\ R_t\in\dd y,\
			\cS_t\in\dd\cS,\ t<T\bigr)
			=
			\frac{p_t^+(\ell,x)p_t^+(r,y)}{Z_\pi(\ell,r)}
			\Smeas_\pi(x,y;\dd\cS)\dd x\,\dd y.
		\end{equation}
		Moreover, on $\{t<T\}$, conditionally on $(L_t,R_t)$,
		the unexplored sector $\cS_t$ is independent of
		$(\cA_L,\cA_R)$, and hence of $\cE_t$.
		
		The surface $\cE_t$ determines every earlier
		$\cE_s$, so on $\{t<T\}$, $\cF_t^0$ is generated by
		$\cE_t$. The conditional law of $\cS_t$
		given $\cF_t^0$ is therefore $\P_{L_t,R_t}$. Since the subsequent exploration
		is determined by $\cS_t$, this gives the Markov property
		for $(L_t,R_t)$.
		Taking masses in \eqref{eq-surface-markov} gives the transition kernel 
		\eqref{eq-tagged-kernel}. 
\end{proof}

We now compute the generator of $(L_t, R_t)$, and show that it extends from all dyadics to $[0,T)$ as a c\`adl\`ag process, along with other basic properties.

\begin{lemma}
	\label{lem-tagged-paths}
	For every $(\ell,r)\in(0,\infty)^2$, the process with transition
	kernel \eqref{eq-tagged-kernel} has generator
	\begin{equation}\label{eq-tagged-generator}
			\cL f(\ell,r)
			=\PV\int_{-\ell}^\infty
			\frac{(\ell+r)^2[f(\ell+z,r)-f(\ell,r)]}
			{\pi z^2(\ell+r+z)^2}\dd z
			+\PV\int_{-r}^\infty
			\frac{(\ell+r)^2[f(\ell,r+z)-f(\ell,r)]}
			{\pi z^2(\ell+r+z)^2}\dd z
	\end{equation}
	for $f\in C_c^\infty((0,\infty)^2)$.
	It admits a c\`adl\`ag version before its lifetime $T$,
	whose coordinates almost surely have no simultaneous
	nonzero jumps before $T$.
	Finally, for every $\delta>0$, the stopping time $
	\tau_\delta=\inf\{0\le t<T:L_t+R_t\le\delta\}$
	satisfies $\tau_\delta<T$ almost surely.
\end{lemma}

\begin{proof}
	Let $X$ be a pair of independent standard symmetric Cauchy
	processes started from $(\ell,r)$ and killed on leaving
	$(0,\infty)^2$, with generator $\cA$. 
	By \eqref{eq-tagged-kernel} and \eqref{eq-half-sector-partition-function}, the length
	process is the Doob $h$-transform of this pair, where
	$h(\ell,r)=Z_\pi(\ell,r)=(\ell+r)^{-2}$.
	Extending $h$ by zero outside the quadrant and at the
	cemetery state, direct integration gives $\cA h=0$
	in the quadrant. The Doob-transform formula
	$\cL f=h^{-1}\cA(hf)$ then yields
	\eqref{eq-tagged-generator}.
	
	This construction gives a c\`adl\`ag version of $(L,R)$
	before its lifetime $T$. For each $u>0$, its law on $[0,u]$,
	restricted to $\{u<T\}$, is absolutely continuous with
	respect to that of the killed Cauchy pair. Thus its
	coordinates almost surely have no simultaneous nonzero
	jumps on $[0,u]$ on $\{u<T\}$. Taking all rational $u$
	proves the assertion before $T$.
	
	For the final assertion, let $\zeta$ be the lifetime of the killed
Cauchy pair $X$. By Dynkin's formula and $\cA h=0$,
the process $M_t=h(X_t)$ is a nonnegative local martingale which is 
equal to zero after $\zeta$. Set
$\sigma_\delta=\inf\{t\ge0:M_t\ge\delta^{-2}\}
=\inf\{0\le t<\zeta:X_t^1+X_t^2\le\delta\}$.
The Doob $h$-transform identity gives
\[
\E\!\left[
h(X_{\sigma_\delta});\,\sigma_\delta<\zeta
\right]
=h(\ell,r)\,\P_{\ell,r}(\tau_\delta<T)
\le h(\ell,r).
\]
Hence $h(X_{\sigma_\delta})
\mathbf1_{\{\sigma_\delta<\zeta\}}$ is integrable.
This and the bound $M_t<\delta^{-2}$ for $t<\sigma_\delta$
show that $(M_{t\wedge\sigma_\delta})$ is a uniformly
integrable martingale. Optional stopping at $\zeta$ therefore
gives
\[
h(\ell,r)
=\E[M_{\zeta\wedge\sigma_\delta}]
=\E\!\left[
h(X_{\sigma_\delta});\,\sigma_\delta<\zeta
\right].
\]
Comparing the two equations yields
$\P_{\ell,r}(\tau_\delta<T)=1$.
\end{proof}

Henceforth, $(L_t,R_t)$ denotes this c\`adl\`ag extension.
We set $L_t=R_t=0$ for $t\ge T$.

\subsection{Strong Markov property}
\label{subsec-strong-markov}
 Section~\ref{subsec-vertical-markov} only dealt with dyadic times, at which no jumps occur a.s. In this section we extend the Markov property to all stopping times before $T$. 

First, we give the setup needed to describe what happens, geometrically speaking, at jumps.
We work in the setting of Section~\ref{subsec-vertical-markov}. Embed the flat surface $\cS_0 \sim \P_{\ell, r}$ in $\D_+$ so its vertex is $0$ and its straight sides are $[-1,0]$ and $[0,1]$. Let $F$ be its developing map normalized by $F(0)=0$
and $F'(0)>0$. As before let $(p_t)_{[0,T)}$ be its angle-$\pi/2$ flow line from 0, so $F(p_t) = i t$.  
 For dyadic $0<q < T$ let $\gamma_q:[-1,1] \to \ol {\D_+}$ be the horizontal chord through $p_q$, where we parametrize $\gamma_q$ from left to right by fixing $\gamma_q(0) = p_q$ and specifying that it has constant speed on each of $(-1,0)$ and $(0,1)$. 
 Let $U_q$ and $V_q$ be the connected components of $\D_+ \backslash \gamma_q$, where $V_q$ is the component whose closure contains 0. As before $L_q$ and $R_q$ are the lengths of $\gamma_q|_{(-1,0)}$ and $\gamma_q|_{(0,1)}$, and let $(L_t, R_t)_{[0,T)}$ be the c\`adl\`ag extension to $[0,T)$. 

\begin{proposition}
	\label{prop-all-time-bases}
	For a.e.\ $(\ell,r)\in(0,\infty)^2$, the following holds a.s.\ for $\cS_0 \sim \P_{\ell, r}$. For all $0 < t < T$,
	\begin{equation}\label{eq-one-sided-domains}
		U_{t+}=\bigcup_{\text{dyadic } q \in (t , T)}U_q,
		\qquad
		V_{t-}=\bigcup_{\text{dyadic } q \in (0,t)}V_q
	\end{equation}
	are disjoint Jordan domains. The uniform limits $\gamma_t^+ = \lim_{q \downarrow t} \gamma_q$ and $\gamma_t^- = \lim_{q \uparrow t} \gamma_q$ exist and are simple straight boundary arcs of $U_{t+}$ and $V_{t-}$ respectively. They pass through $p_t$ and have left and right lengths $(L_t, R_t)$ and $(L_{t-}, R_{t-})$ respectively.

For $0<t<T$, let $\psi_t:\D_+\to U_{t+}$ be the conformal map
sending $(-1,0,1) \mapsto (\gamma_t^+(-1),p_t,\gamma_t^+(1))$.
Set $\psi_0=\operatorname{id}$.
Then, for every $0\le t<T$, as $q \downarrow t$ through dyadic times we have 
\[
F\circ\psi_q-i q
\longrightarrow F\circ\psi_t-i t
\qquad\text{uniformly on }\ol{\D_+}.
\]
\end{proposition}

See Figure~\ref{fig-jumps} (middle, right) for an illustration with $U_{t+}$ and $V_{t-}$ drawn in pink and blue (with $\gamma_t^+$ and $\gamma_t^-$ in red and blue). We emphasize that Proposition~\ref{prop-all-time-bases} does not prove important geometric properties depicted in Figure~\ref{fig-jumps} (e.g., at most one child is born at each time); these are shown in Section~\ref{subsec-jump-geometry}. The proof of Proposition~\ref{prop-all-time-bases} is deferred to  Appendix~\ref{subsec-limiting-arcs}.

For $0<  t<T$, let $\cS_t$ be the flat surface
$U_{t+}$ equipped with the metric inherited from $\cS_0$,
with base $\gamma_t^+$ and vertex $p_t$.
This extends the definition at dyadic times.
Its developing map in half-disk coordinates is
$F\circ\psi_t-i t$.

Let $(\cF_t)$ be the usual augmentation of
$(\cF_t^0)$, obtained by adjoining all subsets of
$\P_{\ell,r}$-null sets and making the filtration
right-continuous. A standard argument allows us to extend  Proposition~\ref{prop-retained} to arbitrary stopping times $\tau < T$. 

\begin{proposition}
	\label{prop-stopped}
	For a.e.\ $\ell, r > 0$, under
	$\P_{\ell,r}$, for any $(\cF_t)$-stopping time $\tau$, on $\{0 <  \tau < T\}$ the conditional law of $\cS_\tau$ given $\cF_\tau$  is $\P_{L_\tau, R_\tau}$. 
\end{proposition}

\begin{proof}
		Write $F_t=F\circ\psi_t-it$ for the developing
		map of $\cS_t$.
		By Definition~\ref{def-conditioned-sector-law}, the law of the normalized developing map
		under $\P_{\ell,r}$ varies weakly continuously with
		$(\ell,r)$ in $\Hol(\D_+)$.
	
	Set 
	$\tau_n=2^{-n}\bigl(\lfloor2^n\tau\rfloor+1\bigr)$,
	with $\tau_n=\infty$ when $\tau=\infty$.
On $\{0 < \tau<T\}$, eventually $\tau_n<T$, and
Proposition~\ref{prop-all-time-bases} and right continuity of the lengths give
\[
F_{\tau_n}\to  F_\tau
\quad\text{in }\Hol(\D_+) \quad \text{ and } 
(L_{\tau_n},R_{\tau_n})\to(L_\tau,R_\tau) \quad \text{a.s.}
\]
This and weak continuity of the measures $\P_{\ell, r}$ from Definition~\ref{def-conditioned-sector-law}
	allow us to extend the dyadic time conditional
	law in Proposition~\ref{prop-retained} to the $\cF_\tau$-conditional law by the usual stopping-time
	approximation argument.
\end{proof}

\begin{remark}
	\label{rem-stopped-versions}
	Fix $b>0$. Sample a flat surface with law $\P_b$ and
	sample a uniform point on its base as its vertex, and define the vertical exploration as above. Then the conclusions of Propositions~\ref{prop-all-time-bases} and~\ref{prop-stopped}
	hold for this exploration, and its length process has transition
	kernel \eqref{eq-tagged-kernel}.
	
	Indeed, by Definition~\ref{def-conditioned-sector-law} and the aforementioned results, the claim holds for a.e.\ $b > 0$. The scale symmetry in Definition~\ref{def-conditioned-sector-law} then yields the claim for all $b>0$. 
\end{remark}

\section{Branching exploration of a JT sector}
\label{sec-branching}

In this section, we extend the strong Markov property for the vertical exploration of an angle-$\pi$ JT sector to a full branching exploration. As before, we draw horizontal chords through points on a vertical flow line --- the explored region is swept out by the union of these horizontal chords, and the unexplored regions are independent angle-$\pi$ JT sectors given their base lengths. After finishing such an exploration, we iterate in each unexplored child, and so forth, yielding a branching exploration.  

\subsection{Geometry of jumps}
\label{subsec-jump-geometry}
All proofs in this section are deferred to  Appendix~\ref{subsec-chords-at-jumps}.
Recall the setup of Remark~\ref{rem-stopped-versions}: fix $b > 0$, sample a JT sector from $\P_b$ with its vertex uniformly sampled from the base, and carry out the vertical exploration from the vertex. The angle-$\pi/2$ flow line from the vertex is $(p_t)_{0 \leq t < T}$, and the limits $\gamma_t^-$ and $\gamma_t^+$ of horizontal chords have lengths $(L_{t-}, R_{t-})$ and $(L_t, R_t)$ respectively (Proposition~\ref{prop-all-time-bases}). Recall that $\gamma_t$ denotes the horizontal chord through $p_t$.

\begin{figure}[ht!]
	\begin{center}
		\includegraphics[scale=0.52]{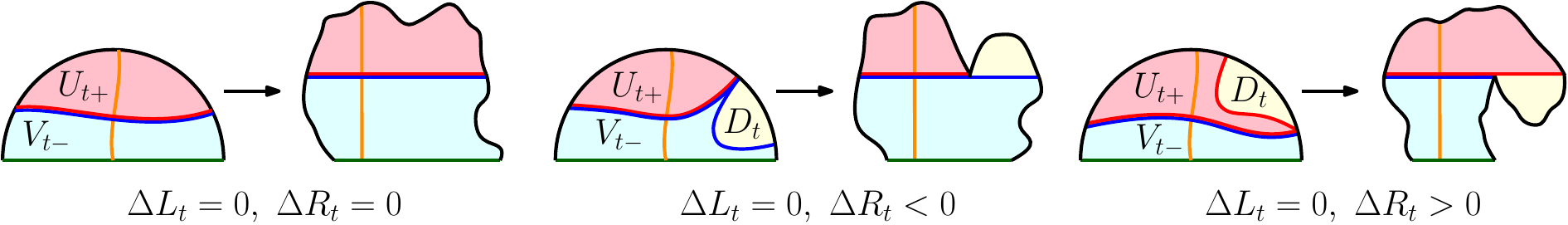}
	\end{center}
	\caption{\label{fig-jumps} 
	The curves $\gamma_t^-$ (blue) and $\gamma_t^+$ (red) are the left- and right-limits of horizontal chords. The orange curve represents the vertical flow line. 
	\textbf{Left:} By Lemma~\ref{lem-actual-minimum}, at non-jump times we have $\gamma_t^- = \gamma_t^+$. \textbf{Middle:} At a negative jump time we have $\gamma_t^+ \subset \gamma_t^-$, and $\gamma_t^\Delta = \ol{\gamma_t^- \backslash \gamma_t^+}$ is the base of a child region $D_t$ (Proposition~\ref{prop-binary-geometric}). \textbf{Right:} At a positive jump time we have $\gamma_t^- \subset \gamma_t^+$, and $\gamma_t^\Delta = \ol{\gamma_t^+ \backslash \gamma_t^-}$ is the base of $D_t$.} 
\end{figure}

\begin{lemma}
	\label{lem-actual-minimum}
	Almost surely, simultaneously for all
	$0<t<T$, we have $\gamma_t \subset \gamma_t^-\cap \gamma_t^+$, and $\gamma_t$ has left and right
	 lengths $
	\min(L_{t-},L_t)$
	and $\min(R_{t-},R_t)$ respectively.
\end{lemma}

At times $t$ at which $(L_t, R_t)$ is continuous, $\gamma_t^-$ and $\gamma_t^+$ both coincide with the horizontal chord through $p_t$, so by Proposition~\ref{prop-all-time-bases} cutting by this horizontal chord gives $U_{t+}$ and $V_{t-}$; no additional connected component is created by the cut (Figure~\ref{fig-jumps}, left). At jump times of $(L_t, R_t)$, exactly one coordinate changes. We next show that cutting along $\gamma_t^- \cup \gamma_t^+$ creates exactly three connected
components: $U_{t+}$, $V_{t-}$, and one additional component which we call a \emph{child component}. This child has its base along $\gamma_t^- \cup \gamma_t^+$, and we say it is \emph{upward} if it locally develops above its base, and \emph{downward} otherwise. 

\begin{proposition}
	\label{prop-binary-geometric}
	Almost surely, the following holds simultaneously
	at every jump time $0<t<T$.
	
	The shorter of $\gamma_t^-$ and $\gamma_t^+$ is the horizontal chord through $p_t$, and the longer is the union of this horizontal chord and an additional horizontal chord $\gamma_t^\Delta$ added on the side whose length jumps. The length of $\gamma_t^\Delta$ is $|\Delta L_t|+|\Delta R_t|$. 
	
	Let $D_t$ be the connected component of $\D_+\setminus \gamma_t^\Delta$ not containing
	$p_t$. It has base $\gamma_t^\Delta$, and is upward if $\Delta L_t+\Delta R_t<0$ and downward if  $\Delta L_t+\Delta R_t>0$.
\end{proposition}

\subsection{The framed JT disk and its symmetries}\label{subsec-framed-disk-symmetries}

\begin{definition}
	Consider a pair $(\cS_1,\cS_2)$ sampled from
	$\iint_0^\infty
	\Smeas_\pi(\ell,r)\otimes\Smeas_\pi(r,\ell)
	\dd\ell\,\dd r$.
	Since their bases are straight, gluing them by arclength
	with opposite boundary orientations produces a flat surface. Let $p$ be the common image of their vertices, orient the gluing seam so $\cS_1$ lies locally to its left, and let $\eta^0_p$ be the forward half of the seam. Let $\cD_\bullet$ be this flat surface with marked point $p$ and curve $\eta^0_p$. We call $\cD_\bullet$ a \emph{framed JT disk}, and denote its law by 
	$\diskone$.
\end{definition}
Recall that the developing map of a flat surface is unique up to
rotation and translation.
Here ``framed'' refers to the inclusion of the curve $\eta_p^0$,
which specifies a preferred rotation when the flat surface is
developed in the plane (namely, $\eta_p^0$ should be mapped to a right-going ray in $\C$). The previous flat surfaces we have worked
with have straight sides which we have used to specify this
preferred rotation.

For $\alpha\in\R$, let $\operatorname{Rot}_\alpha(\cD_\bullet)$
be the same flat surface with marked point $p$ and marked curve
$\eta_p^\alpha$, where angles are measured relative to the
original marked curve $\eta_p^0$.
We extend flow lines to negative times by setting
$\eta_p^\theta(t)=\eta_p^{\theta+\pi}(-t)$ for $t<0$.

\begin{proposition}
	\label{prop-glued-symmetry}
	We have
	$(\operatorname{Rot}_{\pi/2})_*\diskone
	=\diskone$.
	Next, for $t\in\R$ and $\theta\in\{0,\pi/2\}$, let $E_t^\theta$
	be the event that $\eta_p^\theta(t)$ exists.
	Define $T_t^\theta(\cD_\bullet)$ to be $\cD_\bullet$ but having marked point $q=\eta_p^\theta(t)$
	and marked curve $\eta_q^0$ instead.
	Then $T_t^\theta: E_t^\theta\to E_{-t}^\theta$ is a bijection with inverse $T_{-t}^\theta$, and
	\begin{equation}\label{eq-partial-translation}
		(T_t^\theta)_*
		\bigl(\diskone|_{E_t^\theta}\bigr)
		=\diskone|_{E_{-t}^\theta}.
	\end{equation}
\end{proposition}

\begin{proof}
	Cut both sectors perpendicularly from $p$, with seam lengths
	$s$ and $u$.
	Applying \eqref{eq-cutting} twice gives four 
	angle-$\pi/2$ sectors with joint law 
		\begin{equation}\label{eq-four-quarters}
			\int_{(0,\infty)^4}
			\Smeas_{\pi/2}(\ell,s)
			\otimes\Smeas_{\pi/2}(s,r)
			\otimes\Smeas_{\pi/2}(r,u)
			\otimes\Smeas_{\pi/2}(u,\ell)
			\,\dd\ell\,\dd r\,\dd s\,\dd u.
		\end{equation}
		Changing the marked curve from $\eta_p^0$ to
		$\eta_p^{\pi/2}$ cyclically permutes these four sectors
		and their shared side lengths.
	The cyclic symmetry of \eqref{eq-four-quarters} gives
	$(\operatorname{Rot}_{\pi/2})_*\diskone=\diskone$.
	
	Retracing the same flow line shows that
		$T_t^\theta:E_t^\theta\to E_{-t}^\theta$ is a bijection
		with inverse $T_{-t}^\theta$.
	By Proposition~\ref{prop-reroot}, conditional
	on the two sectors after forgetting the marked points
	on their bases, the common marked point is uniform
	with respect to developed arclength on the seam.
	Thus \eqref{eq-partial-translation} holds for $\theta = 0$, and combining with
	$(\operatorname{Rot}_{\pi/2})_*\diskone=\diskone$
	gives the claim for $\theta = \pi/2$.
\end{proof}

For rational $x,y$, let $T_{x,y}=T_x^0\circ T_y^{\pi/2}$ wherever defined. This corresponds to rerooting vertically by $y$ then  horizontally by $x$. By the rerooting symmetry of Proposition~\ref{prop-glued-symmetry} and countability, the following is immediate. 

\begin{lemma}
		\label{lem-reroot-null}
		Consider $\cD_\bullet$  sampled from $\diskone$. Any property which holds a.s.\ for $\cD_\bullet$ also holds simultaneously for all defined rerootings $T_{x,y}(\cD_\bullet)$ with rational $x,y$ a.s. 
	\end{lemma}

\subsection{Conditional law of a child in the vertical exploration}
\label{subsec-child-law}

We now identify the conditional law of a child of the vertical
exploration given the rest of the surface. The result applies
to both upward and downward children. For convenience we parametrize the initial flat surface $\cS_0 \sim \P_b$ in $\D_+$. 

\begin{lemma}
	\label{lem-daughter-law}
	Fix $b > 0$, sample $\cS_0 \sim \P_b$ with its vertex uniformly sampled from the base, and carry out the vertical exploration from the vertex. 
	Fix $\varepsilon>0$ and $n\ge1$, and let $\tau$ be the time of the 
	$n$th jump of $L_t+R_t$ whose magnitude exceeds $\varepsilon$.
	
	On $\{\tau<T\}$, let $D_\tau$ be the unexplored region from
	Proposition~\ref{prop-binary-geometric} with base
	$\gamma_\tau^\Delta$, and let $C$ be the length of 	$\gamma_\tau^\Delta$.
	Conditioned on the flat surface  parametrized by $\D_+ \backslash D_\tau$ having the vertex as its marked point, the conditional law of the flat surface parametrized by $D_\tau$ is  $\P_C$. 
\end{lemma}

\begin{proof}
	By the scaling symmetry of $\P_b$, it suffices to prove the claim for random $b > 0$ (keeping $\varepsilon / b$ fixed). Thus we can work under $\diskone$ instead, treating the ``base'' as the horizontal chord passing through $p$, so the surface lying to the left of $\eta^0_p$ is an angle-$\pi$ JT sector with random base length. 
	
			Fix rational $x,y$, $k\in\{0,1\}$, and a compact interval
			$I\subset(0,\infty)$ with rational endpoints.
			After applying $\operatorname{Rot}_{k\pi}\circ T_{x,y}$ to translate the root and possibly rotate the surface,
			run a \emph{rerooted} vertical exploration, with length process
			$(\wh L_s,\wh R_s)$ and lifetime $\wh T$,
			stopped at
			\[
			\sigma_I=\inf\{s\ge0:\wh L_s+\wh R_s\in I\}.
			\]
			On $\{\sigma_I<\wh T\}$, let $\wh U\subset\D$
		be the union of the unexplored domains at dyadic times
		$s\in(\sigma_I,\wh T)$,
		as in Proposition~\ref{prop-all-time-bases}.
			Let $\wh{\cS}$ and $\wh{\cE}$
			be the flat surfaces parametrized by $\wh U$ and
			$\D\setminus\wh U$, respectively, with the rerooted
			starting point marked in $\wh{\cE}$.
			
			Let $E_{x,y,k,I}$ be the event that the translation $T_{x,y}$ is
			defined, $\sigma_I<\wh T$, and its translation path
			($y$ units vertically, $x$ units horizontally) is disjoint
			from $\wh U$.
			By Propositions~\ref{prop-glued-symmetry} and~\ref{prop-stopped} and Lemma~\ref{lem-reroot-null} and the
			construction of $\diskone$, on $E_{x,y,k,I}$ the
			conditional law of $\wh{\cS}$ given
			$\wh{\cE}$ is
			$\P_{\wh L_{\sigma_I}+\wh R_{\sigma_I}}$.
			Indeed, in the rerooted picture the event $E_{x,y,k,I}$ is
			measurable with respect to $\wh{\cE}$.
			
			Recall that $D_\tau$ is the child specified by $n$ and
			$\varepsilon$, and $C$ is its base length.
			Let $C_\mathrm{ext}$ be the length of $\gamma_\tau^-\cup\gamma_\tau^+$ (i.e., $C_\mathrm{ext} = \max(L_{\tau-} + R_{\tau-}, L_\tau + R_\tau)$),
			so $C_\mathrm{ext}>C$ by Proposition~\ref{prop-binary-geometric}.
			Choose a compact interval $I\subset(0,\infty)$ with rational
			endpoints such that $C\in\operatorname{int}I$ and $\sup I<C_\mathrm{ext}$.
			By Proposition~\ref{prop-all-time-bases}, $\gamma_\tau^-\cup\gamma_\tau^+$ has a one-sided flat
			neighborhood on the exterior side of $D_\tau$.
			Therefore, we can choose rational $x,y$ and $k\in\{0,1\}$
			such that the translation path stays outside $D_\tau$
			and, after applying $\operatorname{Rot}_{k\pi}\circ T_{x,y}$,
			the rerooted vertical exploration reaches the relative
			interior of the child's base from the exterior. Since $L+R$ is c\`adl\`ag, we may choose $y$ sufficiently
			close to $\tau$ that for the rerooted vertical exploration, $\wh L + \wh R$ stays above $\sup I$ until time $\sigma_I$ when it discovers $D_\tau$. 
	
	Choose $(x,y,k,I)$ satisfying the above as a measurable function of
	$(L_t,R_t)_{0\le t\le T}$ (and hence of the flat surface parametrized by $\D \backslash D_\tau$ with the original marked point $p$). Applying the preceding conditional-law statement on each
	of the countably many events specifying the chosen
	$(x,y,k,I)$ gives the desired conditional law. 
\end{proof}

As an immediate application, we observe that all the children of a vertical exploration are conditionally independent JT sectors given $(L_t, R_t)_{[0,T]}$. 

\begin{corollary}
	\label{cor-daughter-history}
	Fix $b>0$, sample a JT sector from $\P_b$,
	and choose its vertex uniformly from its base.
	Run its vertical exploration until it terminates at time $T$.
	
	Conditioned on $(L_t,R_t)_{[0,T]}$, the
	children discovered at jumps are independent, and each child of base length $C$ has conditional law
	$\P_C$.
\end{corollary}

\subsection{Exploration until a cutoff}
\label{subsec-continuum-cutoff}
{Fix $b>\delta>0$.}
Recall from Lemma~\ref{lem-tagged-paths} that in the vertical exploration for $\P_b$, the stopping time $\tau_\delta = \inf \{t \::\: L_t + R_t \leq \delta\}$ satisfies $\tau_\delta < T$ a.s. 
\begin{definition}
	\label{def-jt-exploration-data}
	In the setting of Corollary~\ref{cor-daughter-history}, the
	\emph{vertical exploration data} with cutoff $\delta$ is
	$\bigl((L,R)_{[0,\tau_\delta]},\cC_\delta\bigr)$,
	where
	$\cC_\delta
	=
	\bigl\{(J_t,J_t,\varepsilon_t):
	0<t\le\tau_\delta,\ 
	|\Delta L_t|+|\Delta R_t|>\delta\bigr\}$.
	Here $\varepsilon_t=-\sgn(\Delta L_t+\Delta R_t)$, and
	$J_t=[-L_{t-}+it,-L_t+it]$ if $\Delta L_t\neq0$, or
	$J_t=[R_{t-}+it,R_t+it]$ if $\Delta R_t\neq0$.
	By Proposition~\ref{prop-binary-geometric}, these triples record the
	locations and orientations of the children of base length
	greater than $\delta$.
\end{definition}

The duplicated entry in each of the above triples represents the
fact that each child is attached along its entire base to the
explored region; we include three entries as some discrete models
of interest may have triples $(I,J,\varepsilon)$ with
$I\subsetneq J$.

\begin{remark}
	The law of the vertical exploration data in
	Definition~\ref{def-jt-exploration-data} for a sample from $\P_b$
	can be described explicitly without JT gravity.
	Sample $L_0$ uniformly from $(0,b)$, set $R_0=b-L_0$, and
	run the Doob $h$-transform of independent Cauchy processes
	described in Lemma~\ref{lem-tagged-paths} to get $(L_t,R_t)$,
	stopped at $\tau_\delta$.
	Finally $\cC_\delta$ is a measurable function of
	$(L_t,R_t)$ as in
	Definition~\ref{def-jt-exploration-data}.
\end{remark}

Finally, a \emph{recursive exploration with cutoff $\delta$} is obtained by carrying out the vertical exploration until time ${\tau_\delta}$, then choosing any unexplored child having base length greater than $\delta$ and carrying out a vertical exploration from a uniformly sampled base point, and so forth. The next lemma shows that carrying this out iteratively,
a.s.\ only finitely many steps can occur before reaching a state
where all unexplored children have base length at most $\delta$. 

The proof uses a variant of the quadratic additive martingale in \cite[Proposition~3.7]{ads-growth-fragmentation} with a cutoff. We include full details. 

\begin{lemma}
	\label{lem-continuum-finite-forest}
	Any recursive exploration with cutoff $\delta$ almost surely stops after finitely many steps.
\end{lemma}

\begin{proof}
Let $\cN_\delta$ be the number of exploration steps, including the initial exploration, and let $\cT_\delta$ be the sum of durations of all exploration steps.

We study, for the initial vertical exploration up to time
$\tau_\delta$, the sum $X_t$ of $(L_t+R_t)^2$ and all squared base lengths of all
unexplored children with base length greater than $\delta$. 

		Let $f(\ell,r)=(\ell+r)^2$ and write $w=\ell+r$.
		By \eqref{eq-tagged-generator}, the drift of $(L_t + R_t)^2$ is 
		\[
		\cL f(\ell,r)
		=-\frac{w^2}{\pi}\left(\frac1\ell+\frac1r\right).
		\]
		On the other hand, by  Proposition~\ref{prop-binary-geometric}, a jump of size $z$ creates a child with base length $|z|$.
	Multiplying the jump rates in \eqref{eq-tagged-generator}
	by $z^2\mathbf1_{\{|z|>\delta\}}$ and adding the drift of
	$(L_t+R_t)^2$, the drift of $X_t$ for $t<\tau_\delta$ is
	\[
		\cL f(\ell,r)
		+\frac{w^2}{\pi}
		\left(
		\int_{-\ell}^{\infty}
		\frac{\mathbf1_{\{|z|>\delta\}}}{(w+z)^2}\dd z
		+
		\int_{-r}^{\infty}
		\frac{\mathbf1_{\{|z|>\delta\}}}{(w+z)^2}\dd z
		\right) \le
		-\frac{2w^2}{\pi}\int_0^\delta\frac{\dd z}{(w+z)^2}
		=-\frac{2w\delta}{\pi(w+\delta)}
		\le-\frac{\delta}{\pi}.
	\]
	Here we bounded the omitted contributions below by those
	from positive jumps of size at most $\delta$, and used $w>\delta$.
	Consequently,
	$(X_{t}+(\delta/\pi)t)_{[0,\tau_\delta]}$
	is a nonnegative supermartingale.
	
	By Corollary~\ref{cor-daughter-history}, we can then continue the recursive exploration in any unexplored child of base length exceeding $\delta$, extending the process $X$ to then also include unexplored children from this exploration, and so forth. Across all steps, $X_t + (\delta/\pi)t$ is a nonnegative supermartingale with initial value $b^2$, so the expected total duration of the recursive exploration is bounded by $\tfrac\delta\pi \E[ \cT_\delta] \leq b^2$. 
	
	We now bound $\cN_\delta$ using $\cT_{\delta/2}$. Couple the recursive explorations $\mathfrak E_\delta$ and $\mathfrak E_{\delta/2}$  with cutoffs $\delta$ and $\delta/2$ such that each step of $\mathfrak E_{\delta/2}$ explores the same unexplored child from the same base point as $\mathfrak E_\delta$; after $\mathfrak E_\delta$ terminates continue $\mathfrak E_{\delta/2}$ arbitrarily. 
	Let $c = \E[ \tau_{1/2}\wedge 1]>0$ where $\tau_{1/2}$ is the stopping time for a vertical exploration of a sample from $\P_1$. 
	At each exploration step of $\mathfrak E_\delta$, since the base length of the to-be-explored child is greater than $\delta$, the corresponding vertical exploration step of $\mathfrak E_{\delta/2}$ has expected duration at least $c\delta$. Therefore $c \delta \E [ \cN_\delta] \leq \E[ \cT_{\delta/2}] \leq 2\pi b^2/\delta$. 
	We conclude $\cN_\delta < \infty$ a.s.\ as desired. 
\end{proof}

\section{Convergence of Boltzmann rectilinear disks}
\label{sec-discrete-convergence}

{Recall from Section~\ref{sec-discrete-exploration} that $\P_b^{(B)}$ is the law
of a Boltzmann rectilinear disk with base length $Bb$, with all lengths
divided by $B$. In particular, $\P_b^{(1)}$ is the unscaled law.}

In this section, we check that the hypotheses of Theorem~\ref{thm-exploration-criterion} hold for the Boltzmann rectilinear disks: convergence of the vertical exploration data in Section~\ref{subsec-vertical-convergence}, and tightness  in Section~\ref{subsec-discrete-maximal-width}. 

\begin{proof}[{Proof of Theorem~\ref{thm-main} given Theorem~\ref{thm-exploration-criterion}}]
	The Boltzmann rectilinear disk laws satisfy the branching Markov property by Proposition~\ref{prop-kernel}. 
	The convergence of vertical exploration data is shown in Proposition~\ref{prop-spine}, and tightness of the maximal base length is shown in Lemma~\ref{lem-maximal-width-tightness}. Thus, the convergence of conformal factors follows from  Theorem~\ref{thm-exploration-criterion}.
\end{proof}

\subsection{Convergence of vertical explorations}
\label{subsec-vertical-convergence}

Fix $b>0$. 
Consider sequences $(B_n)$ and $(b_n)$ such that $B_n\to\infty$ and $b_n\to b$ as $n \to \infty$. 
Let $(L_t^{(n)},R_t^{(n)})_{[0,T^{(n)}]}$ be the length process
of the vertical exploration of a sample from $\P_{b_n}^{(B_n)}$
as in Section~\ref{subsec-discrete-vertical}, and let
$(L_t,R_t)_{[0,T]}$ be the corresponding process for a sample
from $\P_b$ as in Section~\ref{subsec-vertical-markov}. As usual, the vertical explorations start from uniformly random points on the bases. 
Their transition kernels are
\eqref{eq-discrete-vertical-kernel} and \eqref{eq-tagged-kernel},
respectively.

Fix $\delta \in (0,b)$ and stop the respective processes at times
\[
\tau_\delta^{(n)}
=\inf\{0 \leq t \leq T^{(n)}\::\:L_t^{(n)}+R_t^{(n)}\le\delta\},
\qquad
\tau_\delta
=\inf\{0\le t<T\::\:L_t+R_t\le\delta\}.
\]
Recall the collection $\cC_\delta^{(n)}$ of triples $(I, J, \varepsilon)$  from Definition~\ref{def-vertical-exploration-data} which identifies the sizes and locations of all unexplored children of the stopped vertical exploration of a sample from $\P_{b_n}^{(B_n)}$ having base length greater than $\delta$. 
We similarly defined $\cC_\delta$ for the stopped vertical exploration of the sample from $\P_b$ (Definition~\ref{def-jt-exploration-data}). Note that in both settings, for every triple $(I, J, \varepsilon)$ we have $I = J$ (but other discrete models of interest could have $I \subsetneq J$).

\begin{proposition}
	\label{prop-spine}
	In the topology of Definition~\ref{def-vertical-exploration-data},
	we have
	\begin{equation}\label{eq-marked-lineage-convergence}
		\left((L^{(n)},R^{(n)})_{[0,\tau_\delta^{(n)}]},
		\cC_\delta^{(n)}\right)
		\Longrightarrow
		\left((L,R)_{[0,\tau_\delta]},
		\cC_\delta\right).
	\end{equation}
\end{proposition}

\begin{proof}
	Throughout the proof, write $B=B_n$. Recall $Z$ and $k$ from Section~\ref{subsec-discrete-vertical}. We record 
	\begin{equation}\label{eq-z-bounds}
		\frac{c}{(1+u)^2}\le Z(u)\le\frac{C}{(1+u)^2}, \qquad k(u)\asymp(1+u)^{-2},
		\qquad
		k(u)\sim\frac1{\pi u^2}\sim2Z(u).
	\end{equation}
The first bound follows from Lemma~\ref{lem-partition}. 
	Splitting the convolution defining $k$ at $v=u/2$ gives the upper bound,
	and dominated convergence with $u-v$ fixed gives the
	asymptotic. The lower bound follows from the asymptotic
	and the atom of $M$ at zero.
	
	\smallskip
	\noindent\textbf{Step 1: Convergence of the unweighted L\'evy processes.}
	Let $X^{(n)}$ be a pair of independent symmetric
	compound Poisson processes, each with L\'evy measure
	$B^2k(B|z|)\dd z$, and let $ X$ be a pair of
	independent standard symmetric Cauchy processes.
	For each $\theta\in\R$, \eqref{eq-z-bounds} and dominated convergence give
	\[
	\int_\R(1-\cos(\theta z))B^2k(B|z|)\dd z
	\longrightarrow
	\int_\R\frac{1-\cos(\theta z)}{\pi z^2}\dd z
	=|\theta|,
	\]
	since $B^2k(B|z|)\le C/z^2$.
	The L\'evy convergence theorem
	\cite[Corollary~VII.3.6]{jacod-shiryaev} therefore gives
	$X^{(n)}\Longrightarrow X$ in the joint $J_1$ topology.
	Couple their increments so that this convergence holds almost
	surely, take an independent $U$ uniform on $(0,1)$, and start
	the processes at $(b_nU,b_n(1-U))$ and $(bU,b(1-U))$,
	respectively. Then $ X^{(n)}\to X$
	almost surely in the joint $J_1$ topology.
	
	Let $S_n,S$ be their respective exit times from $O_\delta:= \{(\ell,r)\in(0,\infty)^2:\ell+r>\delta\}$.
	Almost surely, $S<\infty$, the path before $S$ stays a
	positive distance from $\partial O_\delta$, and the exit
	jump lands outside $\ol{O_\delta}$. Thus the $J_1$ time change identifies $S_n$ and $S$. We conclude 
	\begin{equation}\label{eq-raw-stopped-convergence}
		(X^{(n)})_{[0,S_n]} \longrightarrow (X)_{[0,S]} \quad \text{ almost surely }
	\end{equation}
	in the sense that $S_n \to S$ a.s.\ and $(X^{(n)}(S_n \cdot))_{[0,1]} \to (X(S \cdot))_{[0,1]}$ a.s.\ in the joint $J_1$ topology. 
	
	\smallskip
	\noindent\textbf{Step 2: Convergence of the length processes.}
Let $h_B(\ell,r)=2\pi B^2Z(B(\ell+r))$ and
	$h(\ell,r)=(\ell+r)^{-2}$, both extended by zero outside
	the positive quadrant. By Lemma~\ref{lem-partition},
	$h_B\to h$ locally uniformly in $(0, \infty)^2$.
	Under the preceding coupling, define the densities
	\[
	D_n=\frac{h_B( X^{(n)}_{S_n})}{h_B(X^{(n)}_0)},
	\qquad
	D=\frac{h( X_S)}{h(X_0)}, 
	\]
	 so \eqref{eq-raw-stopped-convergence} gives
	$D_n\to D$ almost surely.
	
	By Lemma~\ref{lem-discrete-h-transform}, weighting the stopped
	path law on the left of \eqref{eq-raw-stopped-convergence}
	by $D_n$ gives the law of 
	$(L^{(n)}, R^{(n)})_{[0,\tau_\delta^{(n)}]}$ 
	restricted to $\{\tau_\delta^{(n)}<T^{(n)}\}$. Similarly, by \eqref{eq-tagged-kernel} weighting the stopped path law on the right of \eqref{eq-raw-stopped-convergence} by $D$ gives the law of $(L, R)_{[0,\tau_\delta]}$. 
In particular,
	$\E[D_n]=\P(\tau_\delta^{(n)}<T^{(n)})\le1$, and, by  Lemma~\ref{lem-tagged-paths}, $\E [D]=\P(\tau_\delta<T)=1$.
	Fatou's lemma gives $\E [D_n]\to1$, and hence
	$D_n\to D$ in $L^1$ by Scheff\'e's lemma. Therefore 
	\begin{equation}\label{eq-stopped-length-convergence}
		(L^{(n)},R^{(n)})_{[0,\tau_\delta^{(n)}]} \Longrightarrow (L, R)_{[0,\tau_\delta]}. 
	\end{equation}

	\smallskip
	\noindent\textbf{Step 3: Convergence of the child attachment data.} At each jump of the process $(L^{(n)}, R^{(n)})$ before $T^{(n)}$, the size of the jump is the sum of the base length of the unexplored child (if nonempty) and the length of a horizontal boundary segment; we first show that these segment lengths tend to zero uniformly in probability up until $\tau_\delta^{(n)}$. 
At a jump before $T^{(n)}$ of magnitude $a$, the conditional law of this horizontal segment length is 
\begin{equation}\label{eq-gap-law}
	\frac{
		B e^{-By/2}Z(B(a-y))\mathbf1_{\{0\le y<a\}}\dd y
		+\q^2e^{-Ba/2}\delta_a(\dd y)
	}{k(Ba)}.
\end{equation}
	By \eqref{eq-z-bounds}, 
	$\frac{Z(B(a-y))}{k(Ba)}\le C(1+By)^2$.
	Including the atom gives
	$\P(y>\varepsilon\mid a)\le Ce^{-cB\varepsilon}$
	uniformly in $a$.
	Since row events occur at rate $B$, for each $t,\varepsilon>0$,
	\[
	\P\bigl(\text{some jump before }t\wedge T^{(n)}
	\text{ has }y>\varepsilon\bigr)
	\le CBt\,e^{-cB\varepsilon}\longrightarrow0.
	\]
Since  $\tau_\delta^{(n)}\Longrightarrow\tau_\delta<\infty$, the maximum of these segment lengths up until time $\tau_\delta^{(n)}$ tends to zero in probability. 
	
	Therefore, up to errors tending uniformly to zero in probability,
	the discrete child intervals and base lengths are given by the same functions of the jumps of $(L^{(n)}, R^{(n)})_{[0, \tau_\delta^{(n)}]}$ as in the continuum setting 
	Definition~\ref{def-jt-exploration-data}, while the orientation formula holds exactly.
	The joint $J_1$ convergence \eqref{eq-stopped-length-convergence} matches the finitely many jumps
	larger than $\delta$. This gives convergence of the
	child records jointly with the stopped length paths. 
	With \eqref{eq-stopped-length-convergence}, this proves convergence of the full exploration data.
\end{proof}

\subsection{Maximal widths}
\label{subsec-discrete-maximal-width}

We verify the second hypothesis of Theorem~\ref{thm-exploration-criterion}
by bounding the base lengths throughout the complete row-by-row
exploration with no cutoff. 

The proof of Lemma~\ref{lem-width} uses a discrete analogue of the martingale 
of \cite[Proposition~3.7]{ads-growth-fragmentation}.

\begin{lemma}\label{lem-width}
	For $R\ge b>0$, let $E_R$ be the event that some explored disk  in the complete row-by-row exploration  has base length at least $R$. Then 
$\P_b^{(1)}
(E_R)
		\le Z(R)/Z(b)$.
\end{lemma}

\begin{proof}
	We use the transition kernels from  Proposition~\ref{prop-kernel}. 
	Let $b_i>0$ be the {base lengths} of the nonempty children at the
	first row. We claim that
	\begin{equation}\label{eq-linearization}
		\E_b^{(1)}\sum_i Z(b_i)^{-1}
		\le Z(b)^{-1}.
	\end{equation}
	
	For a $\mathrm G$ event, the two child contributions are equal
	by symmetry. Counting the child of {base length} $u>0$ and using that $M(\dd u)=Z(u)\dd u$ on $(0,\infty)$,
	\[
	\begin{aligned}
		Z(b){\E_b^{(1)}}\left[
		\mathbf 1_{\mathrm G}\sum_{i} Z(b_i)^{-1}
		\right] &=
		2Z(b)\int_{[0,b)}
		\int_0^{b-v}
		\frac1{Z(u)}
		\frac{e^{-(b-u-v)/2}}{Z(b)}
		Z(u)\dd u \, M(\dd v)
		\\&=
		4\int_{[0,b)}
		\left(1-e^{-(b-v)/2}\right)M(\dd v) \le 4M([0,b)).
	\end{aligned}
	\]
	For an $\mathrm R$ event, we similarly compute 
	\[
	\begin{aligned}
		Z(b){\E_b^{(1)}}&\left[
		\mathbf 1_{\mathrm R}\sum_i Z(b_i)^{-1}
		\right]
		=
		\int_{[0,\infty)}M(\dd v)
		\int_{b+v}^{\infty}
		e^{-(u-b-v)/2}\dd u
		+
		\int_b^\infty Z(u)\dd u
		\int_0^{u-b}
		e^{-(u-b-v)/2}\dd v\\
		&\qquad \qquad \quad =
		2M([0,\infty))
		+2\int_b^\infty
		\left(1-e^{-(u-b)/2}\right)Z(u)\dd u
		\le 2M([0,\infty)) + 2M([b,\infty)).
	\end{aligned}
	\]
	By reflection, the same formula holds with $\mathrm L$ in place of $\mathrm R$.
	
	Adding the three event contributions gives
	$Z(b){\E_b^{(1)}}\sum_i Z(b_i)^{-1}
	\le 8M([0,\infty))=1$, where we used $M([0,\infty))=1/8$. Thus
	\eqref{eq-linearization} holds.
	
	Finally, carry out the complete exploration by, for instance, always exploring the largest unexplored child next. Let $S_k$ be
	the sum of $1/Z$ over the unexplored children after
	$k$ steps. By Proposition~\ref{prop-kernel} and \eqref{eq-linearization},
	$S_k$ is a nonnegative supermartingale with
	$S_0=Z(b)^{-1}$. Since $Z$ is decreasing, {the event
		$E_R$} implies
	$\sup_k S_k\ge Z(R)^{-1}$. Doob's maximal inequality therefore gives
	\[
	{\P_b^{(1)}}
(E_R)
	\le
	{\P_b^{(1)}}
	\Big(\sup_{k\ge0}S_k\ge Z(R)^{-1}\Big)
	\le \frac{Z(R)}{Z(b)}. \qedhere
	\]
\end{proof}

\begin{lemma}\label{lem-maximal-width-tightness}
	Let $B_n\to\infty$ and $\cS^{(n)}\sim\P_1^{(B_n)}$.
	Let $W_{\max}^{(n)}$ be the supremum of the lengths of intrinsic
	horizontal segments in the interior of $\cS^{(n)}$,
	with the horizontal direction parallel to its base.
	Then the sequence $(W_{\max}^{(n)})$ is tight.
\end{lemma}

\begin{proof}
	The event $\{W_{\max}^{(n)}\ge R\}$ implies the event $E_{B_nR}$ (defined in Lemma~\ref{lem-width})
	for the unscaled disk. Hence, by Lemma~\ref{lem-width} and
	\eqref{eq-partition}, for $R\ge1$ we have 
		$\P(W_{\max}^{(n)}\ge R)
		\le {Z(B_nR)}/{Z(B_n)}$.
		By Lemma~\ref{lem-partition}, this ratio
		converges to $R^{-2}$ as $n\to\infty$.
		This implies tightness.
\end{proof}

\section{Proof of the convergence criterion}
\label{sec-developing-convergence}
Recall that a developing map of a flat surface
$(\D_+,e^{2\sigma(z)}|dz|^2)$ is a holomorphic function
$F:\D_+\to\C$ satisfying
$\sigma(z)=\log|F'(z)|$ for $z\in\D_+$.
It is unique up to rotation and translation.
	
In this section we prove Theorem~\ref{thm-exploration-criterion}. It is an immediate consequence of a variant phrased in terms of the developing maps. 

\begin{theorem}\label{thm-developing-convergence}
	Suppose $\cS^{(n)}$ and $\cS$ are as in
	Theorem~\ref{thm-exploration-criterion}. Parametrize these flat surfaces
	by $\D_+$ so that the base corresponds to $[-1,1]$ and its
	metric midpoint to $0$. Let $F^{(n)},F:\D_+\to\C$ be their
	developing maps, normalized such that they fix 0 and have positive derivative there. 
	Extend $F^{(n)}$ and $F$  holomorphically to
	$\D$ by Schwarz reflection across the base. Then $F^{(n)}$ converges in law to $F$ with the  topology of locally uniform convergence.
\end{theorem}

Given this we can pass to conformal factors using
${\sigma_{\cS^{(n)}}=\log|(F^{(n)})'|}$ and
$\sigma_{\cS}=\log|F'|$.

\begin{proof}[Proof of Theorem~\ref{thm-exploration-criterion} given Theorem~\ref{thm-developing-convergence}]
		By Theorem~\ref{thm-developing-convergence},
	$F^{(n)}\Longrightarrow F$ in $\Hol(\D)$.
	Since $F'$ is nowhere zero, $|F'|$ is uniformly bounded away from zero on each compact subset of $\D$.
	Cauchy's formula and the continuous mapping theorem then give
	\[
	\log|(F^{(n)})'|\Longrightarrow\log|F'|
	\qquad\text{in }C^\infty_{\mathrm{loc}}(\D).
	\]
	Using
	$\sigma_{\cS^{(n)}}=\log|(F^{(n)})'|$
	and $\sigma_{\cS}=\log|F'|$,
	and restricting to $\D_+\cup(-1,1)$, proves the result.
\end{proof}

\subsection{Tightness from width bounds}
\label{subsec-width-tightness}

We first obtain tightness of the maps $F^{(n)}$. This will allow us to take a subsequential limit, which we will eventually identify with $F$. 

\begin{proposition}
	\label{prop-bloch-width}
	{Let $W_{\max}^{(n)}$ denote the supremum of the lengths of horizontal chords of $\cS^{(n)}$.} There is a universal constant $C>0$ such that 
	\begin{equation}\label{eq-compact-width-bound}
		|{F^{(n)}}(z)|
		\le C W_{\max}^{(n)}\log\frac{1+|z|}{1-|z|} \qquad \text{ for all } z \in \D.
	\end{equation}
	In particular, hypothesis \textup{(ii)} of
	Theorem~\ref{thm-exploration-criterion} (tightness of the $W_{\max}^{(n)}$) implies that $(F^{(n)})$ is tight
	in $\Hol(\D)$.
\end{proposition}

\begin{proof}
	In this proof, $C$ denotes a positive universal constant whose value may change from line to line. 
	Fix $z\in\D$ and any conformal automorphism $\phi: \D \to \D$ such that $\phi(0) = z$. 
	By Bloch's theorem
	\cite[Chapter~XII, Theorem~1.4]{conway-complex-analysis} applied to the holomorphic function ${F^{(n)}} \circ \phi$, 
	there exists an open set $U \subset \D$ which ${F^{(n)}\circ\phi}$ bijectively maps onto a Euclidean disk of radius $|({F^{(n)}} \circ \phi)'(0)| / C$. Therefore
	\[ W_{\max}^{(n)} \geq 2C^{-1} |({F^{(n)}} \circ \phi)'(0)| = 2C^{-1} (1-|z|^2) |{(F^{(n)})'}(z)| . \]
	Finally, using ${F^{(n)}}(z) = \int_0^z {(F^{(n)})'}(w) \dd w$ and the triangle inequality gives \eqref{eq-compact-width-bound}. 
	Montel's theorem and tightness of $(W_{\max}^{(n)})$
	then imply tightness of $({F^{(n)}})$ in $\Hol(\D)$.
\end{proof}

\subsection{Convergence of explored regions}
\label{subsec-explored-region-convergence}

Fix $\delta > 0$. 
In this section, we define a recursive exploration with a cutoff parameter $\delta$. This makes sense both for the discrete and continuum random surfaces. We show that the discrete explored regions converge in law to the continuum explored regions. 

We now introduce a recursive exploration with cutoff $\delta$ of $\cS$. Carry out a vertical exploration of $\cS$ stopped at cutoff $\delta$. {Among the unexplored children attached directly to the
explored region, choose the one with the longest attachment interval.
If this interval has length greater than $\delta$, carry out a vertical
exploration of that child with cutoff $\delta$, and repeat. Stop when
every attachment interval to the explored region has length at most
$\delta$. By Lemma~\ref{lem-continuum-finite-forest}, this takes only
finitely many steps almost surely.} Let $\cV_j$ be the vertical exploration data in step $j$, and let $M<\infty$ be the total number of steps. 

In the same way, one can carry out the recursive exploration with cutoff $\delta$ of $\cS^{(n)}$; let  $\cV_j^{(n)}$ be the vertical exploration data in step $j$, and let $M^{(n)} \leq \infty$ be the number of steps. 
By hypothesis \textup{(i)} of
Theorem~\ref{thm-exploration-criterion}, we can couple $\cS^{(n)}$ and $\cS$ so $\cV_1^{(n)} \to \cV_1$ a.s. In particular, the base length of the next child chosen by this attachment rule also converges; applying hypothesis \textup{(i)} and re-coupling yields $\cV_2^{(n)} \to \cV_2$ a.s., and so forth. Iterating, we conclude that $M^{(n)} = M$ for all large $n$ a.s., and $(\cV_1^{(n)}, \dots, \cV_{M^{(n)}}^{(n)}) \to(\cV_1, \dots, \cV_{M})$ a.s.

For $1 \leq j \leq M$, let $(L_j, R_j)_{[0,h_j]}$ be the length process of $\cV_j$. The interior of the explored region in the $j$th vertical exploration is described by 
\[\Omega_j = \{ x + i t \: : \: 0 < t < h_j \:\text{ and }\: - \min(L_j(t-), L_j(t)) < x < \min (R_j(t-), R_j(t))\}.\]
Similarly, define $\Omega_j^{(n)}$ from $\cV_j^{(n)}$. 
\begin{figure}[ht!]
	\centering
	\includegraphics[scale=0.7]{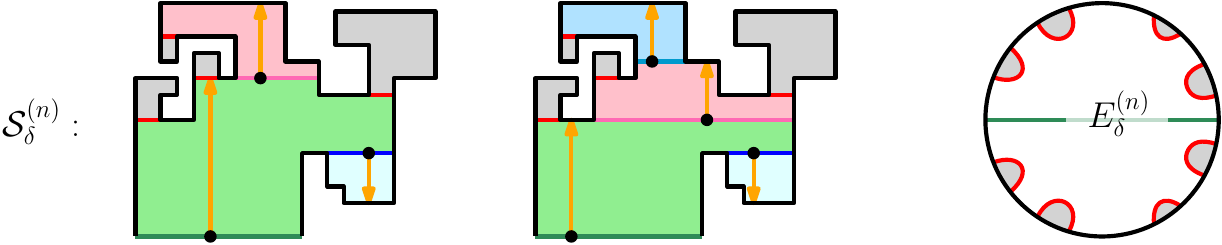}
	\caption{\textbf{Left.} The flat surface $\cS_\delta^{(n)} \subset \cS^{(n)}$ is the gluing of the colorful pieces, which represent the explored regions $\Omega_1^{(n)}, \dots, \Omega_{M^{(n)}}^{(n)}$ of the vertical explorations. The grey regions represent $\cS^{(n)} \backslash \cS_\delta^{(n)}$, and their red boundaries have lengths at most $\delta$.  The randomness in the choice of base points for the vertical explorations may yield different decompositions. \textbf{Right.} We uniformize the double of $\cS^{(n)}$ in $\D$. Then the double of $\cS_\delta^{(n)}$  is parametrized by a region $E_\delta^{(n)} \subset \D$. The red boundary arcs are mapped by $F^{(n)}$ to horizontal segments of length at most $\delta$.}
	\label{fig-core}
\end{figure}

Next, we assemble the explored region of $\cS$ by gluing $\ol {\Omega_1}, \dots, \ol{\Omega_M}$ together. Each of $\Omega_2, \dots, \Omega_M$ corresponds to a child that was explored; for $\Omega_j$ let $(I_j, J_j, \varepsilon_j)$ be its attachment data and $p_j$ the index of the previously explored piece to which it is attached, $j = 2, \dots, M$. Thus  $\ol{\Omega_j}$ should be glued to $\ol{\Omega_{p_j}}$, with $(I_j, J_j, \varepsilon_j)$ describing exactly how the gluing is to be implemented. Thus, we glue  $\ol {\Omega_1}, \dots, \ol{\Omega_M}$ to form a flat surface $\cS_\delta$. Similarly, glue $\ol {\Omega_1^{(n)}}, \dots, \ol{\Omega_M^{(n)}}$ to form $\cS_\delta^{(n)}$. Here $\cS_\delta$ and $\cS_\delta^{(n)}$ are the explored regions of $\cS$ and $\cS^{(n)}$. 

Recall that we have parametrized $\cS$ and $\cS^{(n)}$ in $\D_+$ so their bases correspond to $[-1,1]$ with midpoint at $0$, so their doubles are parametrized by $\D$. The doubles of $\cS_\delta$ and $\cS_\delta^{(n)}$ are then  parametrized by domains $ E_\delta,  E_\delta^{(n)} \subset \D$. Note that by definition, the connected components of  $\partial E_\delta \cap \D$ (resp.\ $\partial  E_\delta^{(n)} \cap \D$) are crosscuts of $\D$ which are mapped by $F$ (resp.\ $F^{(n)}$) to horizontal segments of length at most $\delta$. 
Let 
\begin{equation}\label{eq-explored-region-uniformization}
	\psi_\delta:\D\to E_\delta,
	\qquad
	\psi_{\delta}^{(n)}:\D\to E_\delta^{(n)}
\end{equation}
be the conformal maps fixing $0$ and having positive derivative there; in Section~\ref{subsec-core-exhaustion} we will see that $\psi_\delta$ and $\psi_\delta^{(n)}$ are both close to the identity. 
Set
\begin{equation}\label{eq-explored-region-map}
	H_\delta=F\circ\psi_\delta,
	\qquad
	H_{\delta}^{(n)}=F^{(n)}\circ\psi_{\delta}^{(n)}.
\end{equation}
By definition $H_\delta$ and $H_\delta^{(n)}$ are the developing maps of $\cS_\delta$ and $\cS_\delta^{(n)}$ which fix $0$ and have positive derivative there. 

Since $\cV_j^{(n)} \to \cV_j$, we have 
 the a.s.\ convergence of Hausdorff distances
$d_{\mathrm H}\bigl(
\partial\Omega_j^{(n)},\partial\Omega_j
\bigr)\to 0$, and the gluing data $(I_j^{(n)}, J_j^{(n)}, \varepsilon_j^{(n)})$ converges a.s.\ to $(I_j, J_j, \varepsilon_j)$ for $2 \leq j \leq M$. Crucially, the lengths $|I_j|$ of the attachment intervals are all greater than $\delta$.  This implies
\begin{equation}\label{eq-explored-region-map-convergence}
	\text{under the coupling above, almost surely }\quad  H_\delta^{(n)} \to H_\delta \quad \text{ in } \mathrm{Hol}(\D).
\end{equation}

As a quick consequence, Schwarz's lemma gives
\[\bigl(F^{(n)}\bigr)'(0) \geq \bigl(F^{(n)}\bigr)'(0)
\bigl(\psi_\delta^{(n)}\bigr)'(0) = \bigl(H_\delta^{(n)}\bigr)'(0) > 0.
\]
Under the coupling above,
$\bigl(H_\delta^{(n)}\bigr)'(0)\to H_\delta'(0)>0$
almost surely. Consequently,
\begin{equation}\label{eq-derivative-nondegeneracy}
	\lim_{c\downarrow0}\limsup_{n\to\infty}
	\P\bigl((F^{(n)})'(0)\le c\bigr)=0.
\end{equation}

\subsection{Proof of Theorem~\ref{thm-developing-convergence}}
\label{subsec-core-exhaustion}

We show that, along a convergent sequence of developing maps
with nonconstant limit, letting the exploration cutoff tend
to zero recovers the same limit.
For $r\in(0,1)$, write $\D_r=\{z:|z|<r\}$.

\begin{lemma}\label{lem-short}
	Suppose a sequence of holomorphic functions $f_j: \D\to \C$ converges locally uniformly to a nonconstant holomorphic function $f$. 
	Let $E_j \subset \D$ be domains containing 0 for $j \geq 1$, and let $\delta_j \downarrow 0$. For each $j$, suppose $\partial E_j \cap \D$ is a union of crosscuts $\gamma$ of $\D$ such that $\mathrm{diam}(f_j(\gamma)) \leq \delta_j$. Then for every $0 < r < 1$, we have $\D_r \subset E_j$ for all sufficiently large $j$. 
\end{lemma}

\begin{proof}
	Suppose for the sake of contradiction that for some $0 < r < 1$, along a subsequence we have $\D_r \not \subset E_j$. Let $\gamma_j\subset \partial E_j \cap \D$ be a crosscut intersecting $\D_r$. Since $\mathrm{diam}(f_j(\gamma_j)) \to 0$, passing to a subsequence, $\ol{\gamma_j}$ converges under the Hausdorff metric to a  compact connected set $K$ intersecting $\ol{\D_r}$ and $\partial \D$ such that  $\mathrm{diam}(f(K\cap \D)) = 0$. Thus ${f|_{K\cap\D}}$ is constant, contradicting the fact that $f$ is a nonconstant holomorphic function. 
\end{proof}

\begin{proof}[Proof of Theorem~\ref{thm-developing-convergence}]
	Proposition~\ref{prop-bloch-width} shows the existence of subsequential limits of $F^{(n)}$. 
	Pick any convergent subsequence, relabel the sequence and couple using Skorokhod's representation theorem, so $F^{(n)}$ a.s.\ converges to some holomorphic function $F_*$.  We will show that $F_* \stackrel d= F$. 
	By \eqref{eq-derivative-nondegeneracy} $F_*$ is nonconstant a.s.
	
	The space $\Hol(\D)$ is Polish, hence there is a metric $d$ on the space of probability measures on $\Hol(\D)$ which metrizes weak convergence. 
	
	Fix any deterministic sequence $\delta_j$ with $\delta_j \downarrow 0$. For each fixed $j$, by \eqref{eq-explored-region-map-convergence} $H_{\delta_j}^{(n)}$ converges in law to $H_{\delta_j}$ as $n \to \infty$. Thus we can choose a subsequence $n_j$ such that 
	\begin{equation}\label{eq-metric}
		d(\cL(H_{\delta_j}^{(n_j)}), \cL(H_{\delta_j})) < 1/j \qquad \text{ for all }j, 
	\end{equation}
	where $\cL(X)$ denotes the law of $X$. 
	
	Now, recall the holomorphic maps $\psi_{\delta_j}: \D \to E_{\delta_j}$ and $\psi_{\delta_j}^{(n_j)}: \D \to E_{\delta_j}^{(n_j)}$ from \eqref{eq-explored-region-uniformization}. By Lemma~\ref{lem-short}, for any $0 < r < 1$, we have $\D_r \subset E_{\delta_j}$ and $\D_r \subset E_{\delta_j}^{(n_j)}$ for all sufficiently large $j$. Thus $\psi_{\delta_j}$ and $\psi_{\delta_j}^{(n_j)}$ converge a.s.\ to $\mathrm{id}$ in $\Hol(\D)$. 
	By \eqref{eq-explored-region-map} this implies $H_{\delta_j} \to F$ a.s.
	For any compact $K \subset\D$, picking $r < 1$ such that $K \subset \D_r$, we have for all sufficiently large $j$ that $\psi_{\delta_j}^{(n_j)} (K) \subset {\D_r}$, so 
	\[\sup_{z \in K} |H_{\delta_j}^{(n_j)}(z) - F_*(z)| \leq \sup_{w \in \D_r} |F^{(n_j)}(w) - F_*(w)| + \sup_{z \in K} |F_*(z) - F_*(\psi_{\delta_j}^{(n_j)} (z))| \to 0.\]
	Thus $H_{\delta_j}^{(n_j)} \to F_*$ a.s. Combining with \eqref{eq-metric}, we conclude that $F_* \stackrel d= F$.
\end{proof} 

\appendix

\section{Geometry of cuts of JT surfaces}
\label{app-exploration-geometry}

This appendix proves
Proposition~\ref{prop-all-time-bases}, Lemma~\ref{lem-actual-minimum}, and Proposition~\ref{prop-binary-geometric}.
We first construct the limiting arcs $\gamma_t^\pm$, then establish boundary
estimates and identify the horizontal chords at jumps.

\subsection{Construction of the limiting arcs}
\label{subsec-limiting-arcs}

Recall the conformal coordinates and notation of Section~\ref{subsec-strong-markov}.

\begin{lemma}\label{lem-light-lifts}
	Let $F:\ol{\D_+}\to\C$ be continuous, holomorphic in
	$\D_+$, and satisfy $F'(z)\ne0$ for every $z\in\D_+$.
	If $\gamma_n \subset \D_+$ are flow-line segments for $F$ whose developed
	lengths tend to zero, then
	$\operatorname{diam}\gamma_n\to0$, where diameter is measured
	in the Euclidean metric of $\D_+$.
\end{lemma}

\begin{proof}
	Otherwise, after taking closures and a Hausdorff-convergent
	subsequence, the segments converge to a compact connected set
	$E$ of positive diameter on which $F$ is constant. If $E \cap \D_+ \neq \emptyset$, the holomorphicity of $F$ would force it to be constant, a contradiction. Thus $E$ lies on the boundary and contains a nontrivial
	boundary arc; Schwarz reflection again forces $F$ to be constant, a contradiction.
\end{proof}

\begin{proof}[Proof of Proposition~\ref{prop-all-time-bases}]
	Work on the probability-one event on which $F$ is continuous on
	$\ol{\D_+}$ (Proposition~\ref{prop-boundary-continuity}) 
	and the length process has a c\`adl\`ag extension (Lemma~\ref{lem-tagged-paths}). By construction  $U_q$ decreases while $V_q$ increases
	with $q$.
	Fix $0\le t<T$ and dyadic $q_n\downarrow t$, with $q_n>t$.
	Then $U_{q_n}$ increases to $U_{t+}$ and
	$(L_{q_n},R_{q_n})\to(L_t,R_t)$.
	Note that $L_t,R_t>0$ for all $0\leq t < T$.

	Let $\eta_n = \partial U_{q_n} \cap \partial \D_+$. Monotonicity implies $(\eta_n)$ is an increasing family of boundary arcs; call the limit $\eta$, and give $\eta_n$ and $\eta$ parametrizations such that $\eta_n$ converges uniformly to $\eta$. Next, by definition, 
	\[
	F(\gamma_{q_n}(s))=
	\begin{cases}
		i q_n+L_{q_n}s,&-1\le s\le0,\\
		i q_n+R_{q_n}s,&0\le s\le1.
	\end{cases}
	\]
	Since $L_{q_n},R_{q_n}$ are bounded,
	Lemma~\ref{lem-light-lifts} gives equicontinuity of the family of functions $(\gamma_{q_n})$; pass to a  subsequence along which it converges uniformly to some limit $ \gamma:[-1,1] \to \ol{\D_+}$. 
	It is simple, passes through $p_t$,
	and has left and right lengths $L_t,R_t$. Moreover, monotonicity implies that $\gamma$ intersects $\eta$ only at their endpoints, hence $\eta \cup \gamma$ is a Jordan curve. 
	Along this subsequence, $\eta_n \cup \gamma_{q_n}$ converges uniformly to $\eta \cup \gamma$, hence $U_{t+}$ is a Jordan domain having $\eta \cup \gamma$ as its boundary. This implies that all subsequential limits $\gamma$ agree; call this unique limit 
	$\gamma_t^+$. 
	Note that at each dyadic $0 <q<T$, right continuity of the lengths gives $F\circ\gamma_q^+=F\circ\gamma_q$; this and local injectivity of $F|_{\gamma_q}$ imply  $U_{q+}=U_q$.
	
The uniform convergence of $\gamma_{q_n}$ and the chosen
parametrizations of $\eta_n$ gives uniformly convergent
parametrizations of $\partial U_{q_n}=\gamma_{q_n}\cup\eta_n$
to $\partial U_{t+}=\gamma	_t^+\cup\eta$.
By \cite[Theorem~2.11]{pommerenke-boundary}, we conclude that
$\psi_{q_n}\to\psi_t$ uniformly on $\ol{\D_+}$. Thus, writing $F_s=F\circ\psi_s-is$, continuity of $F$ gives $F_{q_n}\to F_t$
 uniformly on $\ol{\D_+}$.
Schwarz reflection and Hurwitz’s theorem imply that \(F_t'\) is nonzero on $(-1,1)$, so $\gamma_t^+$ is a straight boundary arc for $U_{t+}$
 with its inherited flat metric.
 
 Similar arguments show the claims for $V_{t-}$ and $\gamma_t^-$, and $V_{t-} \cap U_{t+} = \emptyset$ follows immediately from the fact that $V_q\cap U_{q'}=\varnothing$ whenever $q<t<q'$.
The argument applies to every $t$ under the same deterministic
	hypotheses, giving the simultaneous assertion.
\end{proof}

\subsection{Rough-boundary geometry}
\label{subsec-rough-boundary}
In this section we rule out pathological behavior of the ``rough boundary'' $\partial \D_+ \backslash [-1,1]$. 
Take $F=F_{\pi,G,a}$ from \eqref{eq-primitive}.
All almost-sure assertions concern $G\sim\mu$ and hold for every $a\in\R$.

\begin{lemma}\label{lem-no-shelf}
	Almost surely no nontrivial open subarc of $\partial\D_+\setminus[-1,1]$ is mapped by $F$ into a straight segment.
\end{lemma}

\begin{proof}
	For each fixed $\zeta\in\partial\D_+\setminus[-1,1]$,
	$\operatorname{Re}G(r\zeta)$ is centered Gaussian with variance
	$\Var(\operatorname{Re}G(r\zeta))\to\infty$ as $r\uparrow1$.
	Consequently, almost surely $\operatorname{Re}G(r\zeta)$ has no
	finite limit as $r\uparrow1$. This holds simultaneously for all
	$\zeta$ in a fixed countable dense set.
	
	By Proposition~\ref{prop-boundary-continuity}, $F$ is continuous on
	$\ol{\D_+}$. If $F$ mapped an open  arc of $\partial\D_+\setminus[-1,1]$ into a straight segment,
	Schwarz reflection would extend $F$ holomorphically through
	that arc. Since $F$ is nonconstant, its extension satisfies
	$F'(\zeta)\ne0$ at some $\zeta$ in the dense set. But then
	\[
	\operatorname{Re}G(r\zeta)
	=\log|F'(r\zeta)|+\frac12\log|1-r^2\zeta^2|-a
	\]
	has a finite limit as $r\uparrow1$, a contradiction.
\end{proof}

\begin{lemma}\label{lem-inverse}
	Almost surely
	\begin{equation}\label{eq-inverse-small}
		\lim_{\rho\uparrow1}
		\sup_{\substack{z\in\D_+\\|z|>\rho}}
		\frac{1-|z|}{|F'(z)|}=0.
	\end{equation}
\end{lemma}
\begin{proof}
	Because $-G$ has law $\mu$, Proposition~\ref{prop-boundary-continuity}
	implies that the primitive
	$J(z)=\int_0^ze^{-G(w)}(1-w^2)^{-1/2}\dd w$
	extends continuously to $\ol{\D_+}$ almost surely.
	Schwarz reflection extends $J$ holomorphically to $\D$
	and continuously to $\ol\D$.
	If $\omega_J$ is its uniform modulus of continuity, Cauchy's formula
	on the circle of radius $(1-|z|)/2$ about $z$, applied after subtracting
	$J(z)$, gives
	\[
	(1-|z|)|J'(z)|
	\le 2\omega_J((1-|z|)/2)\longrightarrow0
	\qquad\text{as }|z|\uparrow1.
	\]
	Now use
	$1/F'=e^{-a}(1-z^2)J'$ and $|1-z^2|\le2$.
\end{proof}

An embedded Euclidean $2\pi$ sector at a rough-boundary point $\zeta$
means an injective holomorphic $\psi:\varepsilon\D_+\to\D_+$ with
$\psi(iy)\to\zeta$ as $y\downarrow0$ and
$F(\psi(w))=F(\zeta)+e^{i\alpha}w^2$ for some $\varepsilon>0$ and
$\alpha\in\R$. Its developed image is a slit disk with tip $F(\zeta)$.

\begin{proposition}
	\label{prop-no-2pi}
	Almost surely the angle-$\pi$ JT sector has no embedded Euclidean $2\pi$ sector.
\end{proposition}

\begin{proof}
	Work on the probability-one event of Lemma~\ref{lem-inverse}.
	Suppose that $\psi:\varepsilon\D_+\to\D_+$ defines an
	embedded Euclidean $2\pi$ sector at a rough-boundary point
	$\zeta$, so that
	\[
	F(\psi(w))=F(\zeta)+e^{i\alpha}w^2,
	\qquad \psi(iy)\to\zeta \quad\text{as }y\downarrow0.
	\]
	Applying Schwarz--Pick to
	$u\mapsto\psi(ih(1+u))$, for fixed $0<h<\varepsilon/2$,
	gives $1-|\psi(iy)|\ge cy$ for some $c>0$ and all
	$0<y<h$. Differentiating the displayed identity and using
	\eqref{eq-inverse-small}, we obtain
	\[
	|\psi'(iy)|
	=\frac{2y}{|F'(\psi(iy))|}
	\le\frac2c\,
	\frac{1-|\psi(iy)|}{|F'(\psi(iy))|}
	\longrightarrow0.
	\]
	Integrating along the imaginary axis now gives
	\[
	cy\le1-|\psi(iy)|
	\le|\psi(iy)-\zeta|
	\le y\sup_{0<s\le y}|\psi'(is)|
	=o(y),
	\]
	a contradiction.
\end{proof}

\subsection{Horizontal chords at jump times}
\label{subsec-chords-at-jumps}

Work on the common full-measure event of
Lemma~\ref{lem-tagged-paths}, Proposition~\ref{prop-all-time-bases}, Lemma~\ref{lem-no-shelf} and Proposition~\ref{prop-no-2pi}.
Reparametrize $\gamma_t^-$ and $\gamma_t^+$ on $[-L_{t-},R_{t-}]$ and $[-L_t,R_t]$, respectively,
so that $F(\gamma_t^\pm(x))=x+it$.

\begin{proof}[Proof of Lemma~\ref{lem-actual-minimum}]
We have 
	$\gamma_t\subset\gamma_t^-\cap\gamma_t^+$, so $\gamma_t$ has left and right
	lengths bounded above by $
	\min(L_{t-},L_t)$
	and $\min(R_{t-},R_t)$.
	
If this inequality were strict on either side, the endpoint
$\zeta$ of $\gamma_t$ on that side would lie in the relative
interiors of both limiting arcs.
Since $\operatorname{Im}F(\zeta)=t>0$ and $F$ is real on the
original base, $\zeta$ lies in $\partial \D_+ \backslash [-1,1]$.
By Proposition~\ref{prop-all-time-bases}, a neighborhood of $\zeta$ in
$\ol{U_{t+}}$ is mapped by $F$ to an upper half-disk,
while a neighborhood in $\ol{V_{t-}}$ is sent to a lower
half-disk, both centered at $F(\zeta)$.
Their interiors are disjoint because $U_{t+}\cap V_{t-}=\varnothing$.
Joining these neighborhoods along a segment of $\gamma_t$
gives an embedded disk slit along a radius, hence a Euclidean
$2\pi$ sector at $\zeta$, contradicting Proposition~\ref{prop-no-2pi}.
Therefore $\gamma_t$ has left and right
lengths equal to $
\min(L_{t-},L_t)$
and $\min(R_{t-},R_t)$.
\end{proof}

Note that although we stated Proposition~\ref{prop-binary-geometric} before Section~\ref{subsec-framed-disk-symmetries}, the proofs in Section~\ref{subsec-framed-disk-symmetries} do not depend on Proposition~\ref{prop-binary-geometric}, so we can use all results from that section in the following proof. 

\begin{proof}[Proof of Proposition~\ref{prop-binary-geometric}]
	We rule out additional boundary contacts by showing that they
	would force simultaneous jumps of both length coordinates in
	a suitably rerooted exploration.
	
	Work under $\diskone$ on the common full-measure event supplied
	by Lemma~\ref{lem-reroot-null}, where all upward and downward
	explorations from defined rational rerootings satisfy
	Lemma~\ref{lem-tagged-paths},
	Proposition~\ref{prop-all-time-bases}, and
	Lemma~\ref{lem-actual-minimum}, and the boundary satisfies
	Lemma~\ref{lem-no-shelf}.
	
	Without loss of generality fix a jump time $t$ with $\Delta R_t\ne0$; 
	by Lemma~\ref{lem-tagged-paths}, $\Delta L_t=0$, so
	Lemma~\ref{lem-actual-minimum} implies that $\gamma_t$ is the shorter of $\gamma_t^\pm$. Let $\gamma$ denote the longer arc and set
	$\gamma_t^\Delta=\ol{\gamma\setminus\gamma_t}$.
	We need to show that $\gamma_t^\Delta$ meets
	$\partial\D_+\setminus[-1,1]$ only at its endpoints.
	
Let $\zeta$ be the common endpoint of $\gamma_t$ and
$\gamma_t^\Delta$. Suppose that, away from its endpoints, $\gamma_t^\Delta$
contains a point $\xi \in \partial\D_+ \backslash [-1,1]$.
By Lemma~\ref{lem-no-shelf}, the open subarc of
$\gamma_t^\Delta$ between $\zeta$ and $\xi$ meets $\D_+$, so there exists a horizontal chord $\eta$ lying in the relative interior of $\gamma$. 

Suppose first that $\Delta R_t>0$, so $\gamma=\gamma_t^+$.
The flat half-disk neighborhoods in $U_{t+}$ along $\gamma$
allow us to choose a rational rerooting at a point $p'$
just above the interior of $\eta$.
Let $(\wh p_s)$ be the downward vertical flow line
from $p'$, with length process $(\wh L_s,\wh R_s)$
and the limits of horizontal chords
$\wh\gamma_s^\pm$ from above and below.
Let $\sigma$ be the time it reaches $\eta$, so that $\eta$
is the horizontal chord through $\wh p_\sigma$. As $\eta$ lies in the relative interior
of $\gamma$, the flat neighborhoods above $\gamma$ imply 
that $\wh\gamma_\sigma^-$ extends strictly beyond
both endpoints of $\eta$.
Lemma~\ref{lem-actual-minimum} then gives
$\wh\gamma_\sigma^+=\eta$, so $\wh L$ and
$\wh R$ both jump downward at $\sigma$, contradicting
Lemma~\ref{lem-tagged-paths}.

If $\Delta R_t<0$, then $\gamma=\gamma_t^-$. A similar argument gives, after rerooting, a simultaneous downward jump of the coordinates of the length process, again contradicting
Lemma~\ref{lem-tagged-paths}. 
	
	Consequently, $\gamma_t^\Delta$   intersects $\partial \D_+ \backslash [-1,1]$ only at its endpoints, and has length $|\Delta R_t|$.
	Let $D_t$ be the connected component of
	$\D_+\setminus\gamma_t^\Delta$ not containing $p_t$.
	When $\Delta R_t>0$, the component containing $p_t$ contains
	$U_{t+}$, which develops locally above $\gamma_t^\Delta$;
	hence $D_t$ is downward.
	When $\Delta R_t<0$, it contains $V_{t-}$, which develops
	locally below $\gamma_t^\Delta$; hence $D_t$ is upward.
\end{proof}

\bibliographystyle{hmralphaabbrv}
\bibliography{bib}

\end{document}